\documentclass[11pt]{tran-l}

\usepackage{amsfonts,amssymb,verbatim}
\usepackage{amsxtra, amsmath}

\usepackage[varg]{txfonts}

\newcommand\al{\alpha}
\newcommand\bt{\beta}
\newcommand\G{\Gamma}
\newcommand\g{\gamma}

\newcommand\dt{\delta}
\newcommand\e{\varepsilon}
\newcommand\z{\zeta}
\renewcommand\th{\vartheta}
\renewcommand\k{\kappa}
\newcommand\Ld{\Lambda}
\newcommand\ld{\lambda}
\newcommand\x{\xi}
\newcommand\p{\pi}
\renewcommand\r{\rho}
\newcommand\s{\sigma}

\newcommand\ph{\varphi}
\newcommand\ch{\chi}

\newcommand\Om{\Omega}

\newcommand\Ccal{\mathcal{C}}

\newcommand\Ocal{\mathcal{O}}
\newcommand\Pcal{\mathcal{P}}

\newcommand\mm{\mathbf{m}}

\newcommand\pp{\mathbf{p}}

\newcommand\vv{\mathbf{v}}
\newcommand\xx{\mathbf{x}}
\newcommand\yy{\mathbf{y}}

\newcommand\Q{\mathbf{Q}}
\newcommand\R{\mathbf{R}}

\newcommand\CC{\mathbb{C}}
\newcommand\NN{\mathbb{N}}
\newcommand\QQ{\mathbb{Q}}
\newcommand\RR{\mathbb{R}}
\newcommand\ZZ{\mathbb{Z}}

\newcommand\re{\operatorname{Re}}
\newcommand\im{\operatorname{Im}}
\newcommand\sign{\operatorname{sign}}

\newcommand\isdef{\mathrel{:\mskip1mu=}}
\newcommand\oh{O}
\newcommand\vol{\mathrm{vol}}
\renewcommand\setminus{\smallsetminus}

\newcommand\lbs{{\ld_0}}
\newcommand\nubs{{\nu_0}}

\newcommand\pl{\mathrm{Pl}}
\newcommand\plf{\mathrm{pl}}
\newcommand\npl{\widetilde{\mathrm{Pl}}}
\newcommand\nplf{\widetilde{\mathrm{pl}}}
\newcommand\SL{\mathrm{SL}}
\newcommand\PSL{\mathrm{PSL}}

\newcommand\dist{\mathrm{dist\,}}

\newcommand\K{\mathrm{K}}
\newcommand\B{\mathrm{B}}

\newcommand\supp{\mathrm{Supp\,}}
\newcommand\ctt[1]{\mathrm{N}^{#1}}
\newcommand\ct{\ctt{r}} 
\newcommand\nctt[1]{\mathrm{\tilde N}^{#1}}
\newcommand\nct{\nctt{r}} 
\newcommand\nctr{\nct} 
\newcommand\ectr{\mathrm{Eis}^r} 
\newcommand\sg{\mathbf{s}}

\newcommand\prm{\mathfrak p}
\newcommand\kls[1]{S_{\!#1}}

\newcommand\divides{\mathrel{|}}
\newcommand\dividesnot{\mathrel{|\mskip-6mu/}}
\newcommand\Dtfct{\frac{2\sqrt{|D_F|}}{(2\pi)^d}}
\renewcommand\v{V}
\newcommand\nv{\tilde V}

\newtheorem{thm}{Theorem}[section]
\newtheorem{prop}[thm]{Proposition}
\newtheorem{cor}[thm]{Corollary}
\newtheorem{lem}[thm]{Lemma}

\theoremstyle{definition}
\newtheorem{defn}[thm]{Definition}

\makeatletter
\newcommand\matc[4]{\left( {#1\@@atop #3}{#2\@@atop #4}\right)}
\newcommand\matr[4]{\left( {\hfill #1\@@atop\hfill #3}{\hfill
#2\@@atop\hfill #4}\right)}
\newcommand\hmatc[4]{\left[ {#1\@@atop #3}{#2\@@atop #4}\right]}
\newcommand\hmatr[4]{\left[ {\hfill #1\@@atop\hfill #3}{\hfill
#2\@@atop\hfill #4}\right]}
\newcommand\vect[2]{\left( {#1\@@atop #2} \right)}
\newcommand\vectc[2]{\left( {\hfill #1\@@atop\hfill #2} \right)}
\makeatother

\newcommand\summ[2]{ \mathchoice
{\mathop{{\sum}^{#1}}_{#2\hspace*{.2em}}} {\sum^{#1}_{#2}}
{\sum^{#1}_{#2}} {\sum^{#1}_{#2}} }

\newcommand\txtfrac[2]{{\textstyle\frac{#1}{#2}}}

\begin{document}
\title[Density results on Hilbert modular groups II]{Density results
for automorphic forms on Hilbert modular groups II}

\author{Roelof W.\,Bruggeman}
\address{Mathematisch Instituut Universiteit Utrecht, Postbus 80010,
NL-3508 TA Utrecht, Nederland} \email{r.w.bruggeman@uu.nl}

\author{Roberto J.\,Miatello}
\address{FaMAF-CIEM, Universidad Nacional de C\'or\-do\-ba,
C\'or\-do\-ba~5000, Argentina} \email{miatello@mate.uncor.edu}


\begin{abstract}We obtain an asymptotic formula for a weighted sum
over cuspidal eigenvalues in a specific region, for $\SL_2$ over a
totally real number field~$F$, with discrete subgroup of Hecke type
$\G_0(I)$ for a non-zero ideal~$I$ in the ring of integers of~$F$.
The weights are products of Fourier coefficients. This implies in
particular the existence of infinitely many cuspidal automorphic
representations with multi-eigen\-values in various regions growing
to infinity. For instance, in the quadratic case, the regions include
floating boxes, floating balls, sectors, slanted strips (see
\S\ref{sss-rectquad}--\ref{sss-sectquad}) and products of prescribed
small intervals for all but one of the infinite places of~$F$. The
main tool in the derivation is a sum formula of Kuznetsov type
(\cite{BM6+}, Theorem~\ref{thm-sf}).
\end{abstract}

\maketitle

\setcounter{tocdepth}{1}

\tableofcontents

\section*{Introduction} Let $F$ be a totally real number field of
dimension $d$, and let $\Ocal_F$ be its ring of integers. If $I$ is a
non-zero ideal in $\Ocal_F$, let $\G=\G_0(I)$ denote the congruence
subgroup of Hecke type of the Hilbert modular group. We allow a
character of $\G_0(I)$ of the form $\matc abcd \mapsto \ch(d)$, with
$\ch$ a character modulo $I$.

The goal of the present paper is 
to obtain distribution results for cuspidal automorphic
representations of $G\cong \SL(2,\R)^d$ with eigenvalue parameters in
a subset $\Om_t$ of the multi-eigenvalue space, as
$t\rightarrow \infty$, under some general conditions on the family
$\Om_t$.
\smallskip

Let $V_\varpi$ be a cuspidal automorphic representation, with elements
transforming under the above character of $\G_0(I)$, and with a
compatible central character. The Fourier coefficients of automorphic
forms in $V_\varpi$ can be normalized so that they are independent of
the chosen automorphic form in $V_\varpi$. This results in
coefficients $c^r(\varpi)$ describing the Fourier expansion at the
cusp~$\infty$. The Fourier term order $r$ runs through the inverse
different $\Ocal'$ of $F$.

 We denote by $\ld_\varpi = (\ld_{\varpi,j})_j \in \RR^d$ the vector
of eigenvalues of the Casimir operators at the infinite (real)
places of~$F$. For compact sets $ \Om\subset \RR^d$, we consider the
counting functions
\begin{equation}\label{NrOmdef0}
\ct (\Om)\isdef \ct _{\x,\ch}\left( \Om\right)\isdef \sum_{\varpi,\,
  \ld_\varpi \in \Om} |{c^r(\varpi)}|^2. 
\end{equation}
The representations $\varpi$ run through an orthogonal system of
irreducible subspaces of $L^{2,\mathrm{cusp}}_\x(\G_0(I)\backslash
G,\ch)$, for a fixed choice of the character $\ch$ of $\G_0(I)$ and
of the central character (determined by $\x\in \{0,1\}^d$).

The main result in this paper asserts that if the family $t\mapsto
\Om_t$ satisfies certain mild conditions, then
\begin{equation}\label{as}
\ct(\Om_t) = \frac{2\sqrt{|D_F|}}{(2\pi)^d} \pl(\Om_t) + o\left(
\v_1(\Om_t) \right) \qquad(t\rightarrow\infty)
\end{equation}
for all non-zero $r\in \Ocal'$. By $D_F$ we denote the discriminant of
$F$ over~$\QQ$, and by $\pl$ the Plancherel measure of $G$. The error
term contains a reference measure $\v_1$ which, under some general
assumptions, is comparable to $\pl$.

Roughly speaking, we show that
the 
asymptotic formula \eqref{as} holds for the family $t\mapsto \Om_t$
under the conditions that $\Om_t$ grows in at least one coordinate
direction, and that the boundary $\partial \Om_t$ is small in
comparison with $\Om_t$ itself. On the other hand, it is often
convenient to use, instead of $\ld \in \RR^d$, the corresponding
spectral parameter
$\nu_\varpi\in \left( [0,\infty)\cup i\,(0,\infty) \right)^d$. We use
a tilde to indicate that the relevant measures like $\nct$, $\npl$
and $\nv_1$ are taken in the variable~$\nu$, and we write
$$\nct(\tilde\Om_t) = \sum_{\varpi,\nu_\varpi \in \tilde\Om_t} \left|
c^r(\varpi)\right|^2,$$ with $\nu_{\varpi,j} \in [0,\infty) \cup
i\,(0,\infty)$ such that $\ld_{\varpi,j} = \frac14-\nu_{\varpi,j}^2$.

In Theorems \ref{thm-asf1} and~ \ref{thm-asf2} we prove asymptotic
statements in terms of the quantities $\ct(\Om_t)$ and
$\nct(\tilde\Om_t)$ respectively, and this enables us to show
occurrence and density of representations for a wide class of
families of sets $t\mapsto \Om_t$. For illustration, we now list some
of the distribution results that are obtained in the quadratic case.

\begin{enumerate}
\item [(i)] \emph{Small rectangles}. Let $d=2$. Let
$[\al,\bt]\subset [1/4,\infty)$ and consider for $t$ large
$\Om_t = [\al,\bt] \times [t,t+\sqrt t].$ Formula \eqref{as} implies
\begin{equation}
\ct(\Om_t) \sim \frac{\sqrt{D_F}}{2\pi^2}\, \int_\al^\bt \tanh\pi
\sqrt{\ld-\txtfrac14}\, d\ld\;\; t^{1/2}\qquad(t\rightarrow\infty)\,.
\end{equation}
In particular, there are infinitely many $\varpi$ with $\ld_
{\varpi,1}\in [\al,\bt]$. These $\varpi$ have a second component of
unitary principal series type. A similar result holds with $\varpi_2$
of discrete series type.

On the other hand if $[\al,\bt] \subset \left[\lbs,\txtfrac14\right)$
and $\Om_t=[\al,\bt]\times [t,\sqrt t]$, then
\begin{align}
 \ct(\Om_t) &= o\bigl(t^{1/2}\bigr)\qquad(t\rightarrow\infty)\,
 \end{align}
 This gives an upper bound for the weighted density of $\varpi_1$ of
 complementary series type.

As another example, if one takes
$\Om_t = [t,t+\sqrt t]\times [ct,ct+\sqrt t]\,$, with $t$ large and
$c\geq 1$, then
\begin{equation}
\ct(\Om_t) \sim C_{c}\, t\qquad(t\rightarrow\infty).
\end{equation}
\item [(ii)] \emph{Slanting strips}. Let $d=2$, and put, in terms of
the spectral parameter
\[ \tilde \Om_t = \left\{ (\nu_1,\nu_2) \in \left(i[1,\infty)
\right)^2\;:\; t \leq |\nu_1|\leq 2t,\, a|\nu_1|+b \leq |\nu_2| \leq
a|\nu_1|+c\right\}\,, \] with $a>0$, $c>b$ fixed and $t$ large. Then
\[ \nct(\tilde \Om_t) \sim C_{a,b,c}\,
t^3\qquad(t\rightarrow\infty)\,. \]
This shows that we see infinitely many points $\nu_\varpi$ in a
slanted direction. We note that this slanting strip becomes a sector
in $\ld$-space.
\item [(iii)] \emph{Sectors}. Let $d=2$, and fix $0<p<q$,
$\al>\frac12$. For $t$ large put
\[ \Om_t = \biggl\{ (\ld_1,\ld_2)\in [0,\infty)^2\;:\; t \leq
\ld_1\leq t+t^\al\,, p\ld_1\leq \ld_2\leq q\ld_1\biggr\}\,.\] Then we
have
\[ \ct ( \Om_ t) \sim C_{p,q} \, t^{1+\al}
\qquad(t\rightarrow\infty)\,.\]
\item[(iv)] \emph{Spaces of holomorphic cusp forms. }Take
$\tilde\Om_b$ equal to the singleton $b=(b_1,\ldots,b_j)\in \ZZ^d$
with $b_j\geq 1$ and agreeing with the central character, {\sl ie.},
$b_j \equiv \x_j \bmod 2$ for $1\leq j \leq d$. These are the weights
for which there may be holomorphic cusp forms on the product of $d$
copies of the upper half-plane for the character $\ch$ of $\G_0(I)$.
Corollary~\ref{corhol} shows that for some positive constant~$C$ we
have
\begin{equation}\label{holas}
\nct(\tilde\Om_b) \;\sim\; C \; \prod_{j=1}^d
\frac{b_j-1}2\qquad\text{as } \prod_{j=1}^d
\frac{b_j-1}2\rightarrow\infty\,.
\end{equation}

If we take $r$ totally positive, then the quantity
$\nct(\tilde\Om_b) $ has an expression in terms of Fourier
coefficients of holomorphic cusp forms. In particular, this implies
that if there are only finitely many weights $b$ with $b_j\geq 2$ for
all $j$ for which the corresponding space of cusp forms is non-zero.
(We do not obtain information concerning weights $b$ for which $b_j=1$
for some~$j$.)
\end{enumerate}

We will give a much more complete list of applications of the main
asymptotic formula in \S\ref{sss-rectquad}--\ref{sss-sectquad}.
\smallskip

By using the Selberg trace formula (\cite{Se}, \cite{He}), some
unweighted distribution results related to those in this paper have
been obtained in \cite{Hu91} and \cite{HT92}, while results connected
with Weyl's laws in different contexts have been proved by several
authors e.g. \cite{DKV79}, \cite{Do}, \cite{Ef}, \cite{Mi},
\cite{LJ99}, \cite{LaMu}, \cite{Mu}, \cite{LV} and~\cite{LpMu}.

The main tool for the results in the present paper is the Kuznetsov
type sum formula in Theorem 3.21 of~\cite{BM6+}, which we recall in
\S\ref{sect-stat-sf}. It leads to sums weighted by a product of
Fourier coefficients. The results obtained here may be viewed as a
generalization of the results in ~\cite{BMP3a}. To explain the
difference, we note that the sum formula gives a linear relation
between four terms. The two main terms in the present context are a
weighted sum of a test function $\ph$ over the $\nu_\varpi$ and an
integral of $\ph$ against the Plancherel measure. The test function
has a product form $\ph(\nu) = \prod_j \ph_j(\nu_j)$ where $j$ runs
over the infinite places of the totally real number field~$F$. For
the terms that are principal in this paper this product form is
non-essential. However to show that the other terms are small we
need, for one of those terms
(the sum of Kloosterman sums), estimates of a Bessel transformation of
each of the factors~$\ph_j$. This forces the product structure upon
us, in contrast with the case of the Selberg trace formula. There the
integral transformation is the Fourier transform, which respects
rotations.

In \cite{BMP3a} we chose each factor $\ph_j$ as a Gaussian kernel. For
the places in a non-empty set $Q$ of real places, this kernel was an
approximation of the constant function one, for the other places it
was chosen as an approximation of a delta distribution. This led us
to asymptotic results for regions $\Om_X$ of the form
\[\prod_{j\not\in Q} [a_j,b_j] \times \prod_{j\in Q} \bigl\{
(\ld_j)_{j\in Q}\;:\; \sum_{j\in Q}|\ld_j| \leq X\bigr\}\,.\]

The purpose of this paper is to work with sets having a much more
general form in the coordinates in~$Q$. To do this, the test
functions have to be chosen in a much more complicated way. Our
choice is indicated in Lemma~\ref{lem-sf-q}. The idea is to take, for
each place in $Q$ a Gaussian kernel of moderate sharpness. We
approximate the characteristic function of sets in $\prod_{j\in Q}
\left(\RR \cup i\RR\right)$ by a convolution with this Gaussian
kernel. At the real places outside $Q$ we take a general test
function to be specified at a later stage. Application of the sum
formula gives the relation in Proposition~\ref{prop-est1q}. The use
of a Tauberian argument like in \cite{BMP3a} is no longer possible.
To be able to handle the error terms, we first give in
\S\ref{sect-ub} an upper bound for the weighted sums under
consideration. After that we adapt the sharpness of the test
functions to the family of sets we consider.

This leads to Theorem~\ref{thm-asf1}, where we loose control over the
size of the error term and have to be content with an asymptotic
result. This is so because the size of the error term depends on the
family of sets in a complicated way, as equations
\eqref{mr-bt-def}--\eqref{eps-choice} show, and the error term is
almost as large as the main term. The second stage of the method, in
\S\ref{sect-sst}, involves choosing the factors of the test function
at the real places outside $Q$ as approximations of characteristic
functions of intervals in the coordinate $\ld_j$. The central result
is Theorem~\ref{thm-asf2}. It is specialized in \S\ref{sect-special}
to various families of sets that include those mentioned above.
\smallskip

The sum formula involves products
$\overline{c^r(\varpi)}\allowbreak\,c^{r'}(\varpi)$ for two non-zero
Fourier terms orders $r$ and $r'$. Its application in the present
paper works well if $\frac {r'}r$ is totally positive. We intend to
apply the asymptotic results, under this assumption, in subsequent
work where we will take eigenvalues of Hecke operators into account.

\section{Preliminaries and discussion of main results} This section
serves to recall results and fix notations, and after that to state
the main results of this paper.

\subsection{Automorphic representations for Hilbert modular
groups}\label{sect-autrepr} Let $F$ be a totally real number field
with degree~$d$ over $\QQ$. The Lie group $G=\SL_2(\RR)^d$ is the
product $\prod_j \SL_2(k_j)$ over all infinite places $j$ of $F$. We
fix a non-zero ideal~$I$ in the ring of integers $\Ocal$ of $F$. The
group $G$ contains the discrete subgroup
\begin{equation}\label{Gam0}\G_0(I)= \biggl\{ \matc abcd \in
\SL_2(\Ocal)\;:\; c\in I \biggr\}
\end{equation}
with finite covolume.

Let $\ch$ be a character of $\left( \Ocal/I \right)^\ast$. It
determines a character of $\G_0(I)$ by $\ch\matc abcd = \ch(d)$. Let
$L^2(\G_0(I)\backslash G,\ch)$ be the Hilbert space of classes of
functions transforming according to $f(\g g) = \ch(\g)f(g)$ for
$\g\in \G_0(I)$. The group $G$ acts unitarily in this Hilbert space
by right translation. This space is split up according to central
characters, indicated by $\x\in \{0,1\}^d$. By
$L^2_\x(\G_0(I)\backslash G,\ch)$ we mean the subspace on which the
center acts by
\[ \left(\matr{\z_1}00{\z_1},\ldots,\matc{\z_d}00{\z_d} \right)
\mapsto \prod_j \z_j^{\x_j},\] where $\z_j\in \{1,-1\}$. This subspace
can be non-zero only if the following compatibility condition holds:
\begin{equation}\label{ch-x-comp}
\ch(-1) = \prod_j (-1)^{\x_j}.
\end{equation}
We assume this throughout this paper.

   There is an orthogonal decomposition
\begin{equation}\label{L2decomp}
L^2_\x(\G_0(I)\backslash G,\ch) =
L^{2,\mathrm{cont}}_\x(\G_0(I)\backslash G,\ch) \oplus
L^{2,\mathrm{discr}}_\x(\G_0(I)\backslash G,\ch).
\end{equation}
The $G$-invariant subspace
$L^{2,\mathrm{cont}}_\x(\G_0(I)\backslash G,\ch)$ can be described by
integrals of Eisenstein series and the orthogonal complement
$L^{2,\mathrm{discr}}_\x(\G_0(I)\backslash G,\ch)$ is a direct sum of
closed irreducible $G$-invariant subspaces. If $\ch=1$, the constant
functions form an invariant subspace. All other irreducible invariant
subspaces have infinite dimension. They are cuspidal and span the
space $L_\x^{2,\mathrm {cusp}}(\G_0(I)\backslash G,\ch)$, the
orthogonal complement of the constant functions in
$L^{2,\mathrm{discr}}_\x(\G_0(I)\backslash G,\ch)$.

We fix a maximal orthogonal system $\left\{ V_\varpi \right\}_\varpi$
of irreducible subspaces in the Hilbert space
$L^{2,\mathrm{cusp}}_\x(\G_0(I)\backslash G,\ch)$. This system is
 unique if all $\varpi$ are inequivalent. In general, there might be
multiplicities, due to oldforms.
\medskip

Each irreducible automorphic representation $\varpi$ of $G=
\prod_j \SL_2(\RR)$ is the tensor product $\bigotimes_j \varpi_j$ of
irreducible representations of $\SL_2(\RR)$. Here and in the sequel,
$j$ is supposed to run over the $d$ archimedean places of~$F$.

The factor $\varpi_j$ can (almost) be characterized by the eigenvalue
$\ld_{\varpi,j}$ of the Casimir operator of $\SL_2(\RR)$, and the
central character, which is indicated by $\x_j$.

If $\x_j=0$, then the eigenvalue $\ld_{\varpi,j}$ can either be of the
form $\frac b2-\frac{b^2}4$ with $b\geq 2$ even, or else the
$\ld_{\varpi,j}$ have a lower bound $\lbs \in
\left(0,\frac14\right]$.
By the Selberg conjecture, it is expected that one can take
$\lbs=\frac14$. The best result at present is
$\frac{77}{324}=\frac14-\left(\frac19\right)^2 \le \lbs \le \frac14$;
see \cite{KS1}. If $\x_j=1$, then the $\ld_{\varpi,j}$ can either lie
in $\left[ \frac14,\infty\right)$, or be of the form
$\frac b2-\frac{b^2}4$ with $b\geq 3$, $b$ odd. We call
$\ld_{\varpi} = \left( \ld_{\varpi,j} \right) \in \RR^d$ the
eigenvalue vector of the representation $\varpi$.

Spectral theory shows that the set $\{\ld_\varpi\}$ is discrete in
$\RR^d$. To see this we use that the Casimir operator of $G$ has a
discrete spectrum with finite multiplicities in
$L^{2,\mathrm{cusp}}_\x(\G\backslash G)_q$, where
$L^{2,\mathrm{cusp}}_\x(\G\backslash G)_q$ is the subspace of
$L^{2,\mathrm{cusp}}_\x(\G\backslash G)$ with $K$-type (or weight)
$q\in \ZZ^d$, $q\equiv\x\bmod 2$. Hence the number of representations
$\varpi$ (with multiplicities) such that $V_\varpi \cap
L^{2,\mathrm{cusp}}_\x(\G\backslash G)_q\neq\{0\}$ and for which the
eigenvalue $\ld_{\varpi,1}+\ld_{\varpi,2}+\cdots+\ld_{\varpi,d}$ of
the Casimir operator is under a given bound, is finite. For a given
component $\varpi_j$ in the discrete series, the weights $q_j$
occurring in $V_{\varpi,j}$ satisfy $|q_j| \geq b_j\geq 1$, with
$\ld_{\varpi,j} = \frac{b_j}2\bigl(1-\nobreak \frac{b_j}2\bigr)$. So
for a given bounded set $\Om\subset \RR^d$ we can choose the $K$-type
$q$ such that all $\varpi $ with $\ld_\varpi\in \Om$ are present in
$L^{2,\mathrm{cusp}}_\x(\G\backslash G)_q$. Thus $\Om$ contains only
finitely many $\ld_\varpi$, counted with multiplicities.

The correspondence between values of $\ld=\ld_{\varpi,j}$ and
equivalence classes of unitary representations of $\SL_2(\RR)$ of
infinite dimension is one-to-one if $\ld >0$ for $\x=\x_j=0$, and if
$\ld>\frac14$ if $\x=1$. In the other cases, $\ld=\frac
b2-\frac{b^2}4$ with $b\in \ZZ_{\geq 1}$, $b\equiv \x\bmod 2$. In
these cases, there are two equivalence classes, one with lowest
weight $b$ (holomorphic type), and one with highest weight $-b$
(antiholomorphic type). If $b\geq 2$, representations of these classes
occur discretely in $L^2\bigl(\SL_2(\RR)\bigr)$, and are called
\emph{discrete series representations}. The representations in the
case $b=1$ are sometimes called mock discrete series. They do not
occur discretely in $L^2\bigl( \SL_2(\RR) \bigr)$.

All these classes of representations, discrete series or not, may
occur as an irreducible component of
$L^{2,\mathrm{cusp}}_\x\left( \G_0(I)\backslash
G,\ch\right)$.\medskip

 As discussed in \S2.3.4 in \cite{BM6+}, the Fourier expansion of one
automorphic form in $V_\varpi$ determines the Fourier expansion of
any automorphic form in $V_\varpi$. We refer to \cite{BM6+} for the
normalization. This results in coefficients $c^r(\varpi)$ describing
the Fourier expansion at the cusp~$\infty$. The Fourier term order
$r$ runs through the inverse different $\Ocal'$.

In the choice of the $c^r(\varpi)$ there is a freedom of a complex
factor with absolute value one for a given $\varpi$. Since we shall
work with weights $|{c^r(\varpi)}|^2$ this freedom has no influence
on the results we aim at.
\medskip

When dealing with the sum formula, it is technically easier to
parametrize the eigenvalues $\ld_{\varpi,j} \in \RR$ by
$\ld_{\varpi,j}=\frac14-\nu_{\varpi,j}^2$, with, for instance,
$\nu_{\varpi,j}\in (0,\infty) \cup i\,[0,\infty)$. We put $\nu_\varpi
= \bigl( \nu_{\varpi,j}\bigr)$, and call $\nu_\varpi$ and $\x_\varpi=
\bigl( \x_{\varpi,j} \bigr)\in \{0,1\}^d$ the \emph{spectral
parameters} of~$\varpi$.

We have $\nu_\varpi \in Y_\x = \prod_j Y_{\x_j}$, with
\begin{alignat}2 \label{Ydef}
   Y_0&\isdef \left\{ \txtfrac{b-1}2\; :\; b\geq 2\text{ even
   }\,\right\}\;&\cup
   &\;i\,[0,\infty) \cup \left( 0,\nubs \right] ,
\displaybreak[0]
\\
\nonumber Y_1&\isdef \left\{ \txtfrac{b-1}2\;:\; b\geq 3\text{ odd
}\,\right\} &\cup& \;i\,[0,\infty),
\end{alignat}
where $\nubs=\sqrt{\frac14-\lbs}$.

\subsection{Discussion of main results}\label{sect-results}

\subsubsection{Counting function} For compact sets $ \Om\subset
\RR^d$, we use the counting functions
\begin{equation}\label{NrOmdef}
\ct (\Om)\isdef \ct _{\x,\ch}\left( \Om\right)\isdef \sum_{\varpi,\,
\ld_\varpi \in \Om} |{c^r(\varpi)}|^2,
\end{equation}
with $r\in \Ocal'\setminus\{0\}$. The representations $\varpi$ run
through the orthogonal system of irreducible subspaces of
$L^{2,\mathrm{cusp}}_\x(\G_0(I)\backslash G,\ch)$ chosen in
\S\ref{sect-autrepr}. If there are multiplicities larger than one,
the choice of the orthogonal system does not influence the counting.
It may happen that $c^r(\varpi)=0$ for some $\varpi$ and some~$r$.
However, varying $r$ we can detect all~$\varpi$.

More generally, if $f$ is a function on $\RR^d$, we can consider the
sum
\begin{equation}\label{Nrfdef}
\ct(f)\isdef \ct_{\x,\ch}(f) \isdef \sum_\varpi f(\ld_\varpi)
|{c^r(\varpi)}|^2,
\end{equation}
which converges for suitable~$f$, for instance, compactly
supported~$f$. So $\ct(\Om) = \ct(\ch_\Om)$, where $\ch_\Om$ is the
characteristic function of $\Om$.

In the $\nu$-coordinate the counting function is
\begin{equation}\label{tNrfdef}
\nct (\tilde\Om) = \ \sum_{\varpi,\, \nu_\varpi\in \tilde \Om}
|{c^r(\varpi)}|^2,
\end{equation}
for sets $\tilde\Om \subset Y_\x$.

\subsubsection{Plancherel measure}We will compare $\ct(\Om)$ with
$\pl(\Om)=\pl_\x(\Om)$. The measure $\pl$ on $\RR^d$ is the product
$\pl=\bigotimes_j \pl_{\x_j}$, where $\pl_{\x_j}$ is the measure
on~$\RR$ given by
\begin{align}
\pl_0(f) &= \int_{1/4}^\infty f(\ld)
\tanh\pi\sqrt{\ld-\txtfrac14}\, d\ld
\displaybreak[0]\\
\nonumber &\qquad\hbox{} + \sum_{b\geq 2,\, b\equiv 0 \bmod2}
(b-1) f\left( \txtfrac b2\left(1-\txtfrac b2\right) \right),
\displaybreak[0]
\\
\nonumber \pl_1(f)&= \int_{1/4}^\infty f(\ld)
\coth\pi\sqrt{\ld-\txtfrac14}\, d\ld
\displaybreak[0]\\
\nonumber &\qquad\hbox{} + \sum_{b\geq 3,\, b\equiv 1 \bmod2}
(b-1) f\left( \txtfrac b2\left(1-\txtfrac b2\right) \right)
\end{align}
Note that $\pl_{\x_j}$ gives zero measure to the set of
\emph{exceptional eigenvalues} in $\left[ \lbs,\frac14\right)$.

The notations $\pl$ refers to the Plancherel measure of $\SL_2(\RR)$,
see, e.g., \cite{La75}, Chap.~VIII, \S4, p.~174.

In the $\nu$-coordinate the Plancherel measure on $Y_\x$ is given by
$\npl_\x = \bigotimes_j \npl_{\x_j}$, where
\begin{equation}\label{npl-def}
\npl_{\x_j} (f) = 2\int_0^\infty f(it)\, \nplf_j(t) \, dt + 2\sum_{\bt
\in \frac{\x_j+1}2+\NN_0} f(\bt)\, \nplf_j(\bt)\,,
\end{equation}
with
\begin{equation}\label{nplf-def}
\renewcommand\arraystretch{1.3}
\begin{array}{|c|c|c|}
\hline &&\nplf_j(t) \\ \hline \x_j=0 & t\in i\RR& |t|\, \tanh\p |t|
\\ \hline
\x_j=1 & t\in i\RR & |t|\coth \p |t|
\\
 \hline
t\equiv\frac{\x_j-1}2 \bmod 1&t\in \RR\setminus\{0\}&|t|\\ \hline
t\not\equiv\frac{\x_j-1}2\bmod 1&t\in \RR\setminus\{0\}&0\\ \hline
\end{array}
\end{equation}
If $f$ is even, then
\begin{align}
\npl_0(f) &= i \int_{\re\nu=0} f(\nu) \nu \tan\pi\nu \, d\nu+
\sum_{\bt\in \frac12+\ZZ} |\bt|f(\bt),\\
\nonumber \npl_1(f)&= -i \int_{\re\nu=0} f(\nu)\nu\cot\pi\nu\,d\nu
+ \sum_{\bt\in \ZZ\setminus \{0\}} |\bt|f(\bt).
\end{align}
\medskip

The reference measure $\v_1$ is more easily given in the
$\nu$-coordinate. We leave the reformulation in terms of the
$\ld$-coordinate to the reader. The measure has a product form $\nv_1
= \bigotimes_j \nv_{1,\x_j}$ on the space $\left(
(0,\infty) \cup i\,[0,\infty) \right)^d$ with

\begin{align}\label{tV1def}
\int h\, d\nv_{1,0} &= \int_1^\infty t\, h(it)\, dt + \int_0^1 h(it)\,
dt + \int_0^{\nubs} h(x)\, dx + \sum_{\bt>0,\, \bt \equiv
\frac12\,(1)} \bt\, h(\bt)\,,\\
\nonumber \int h\, d\nv_{1,1} &= \int_1^\infty t\, h(it)\, dt +
\int_0^1 h(it)\, dt + \sum_{\bt>0,\, \bt \equiv 0\,(1)} \bt
\,h(\bt)\,.
\end{align}
So $\nv_1$ is positive everywhere on $Y_\x$, and $\npl(\tilde \Om)
\ll \nv_1(\tilde\Om)$ for all $\Om$. We have also $\nv_1(\tilde\Om)
\ll_\e \npl(\tilde\Om)$ if the coordinates of $\tilde\Om_t$ stay away
from $(0,\nubs]\cup i\, [0,\e)$ for some $\e>0$.

\subsubsection{Asymptotic formula}We will prove that for families
$t\mapsto \Om_t$ of sets in $\RR^d$ satisfying conditions discussed
below:
\begin{equation}\label{asf}
\ct(\Om_t) = \Dtfct \pl(\Om_t) +o\bigl( \v_1(\Om_t) \bigr) \quad
(t\rightarrow\infty)\,
\end{equation}
for any $r\in \Ocal'\setminus\{0\}$.

Here $D_F$ is the discriminant over $\QQ$ of the totally real number
field $F$ of degree~$d$. The character $\ch$ of $(\Ocal/I)^\ast$ and
$\x\in \{0,1\}^d$ are as explained in \S\ref{sect-autrepr}.

The sets $\Om_t$ should get large, in particular, $\v_1(\Om_t)$ should
tend to infinity as $t\rightarrow\infty$. Moreover, the boundary
$\partial\Om_t$ should be relatively small in comparison with $\Om_t$
itself. The precise conditions are discussed in \S\ref{sect-U-eps},
Theorem~\ref{thm-asf1}, Proposition~\ref{prop-asf1-u}, and
Theorem~\ref{thm-asf2}.

Since $\pl(\Om_t) \ll \v_1(\Om_t)$ the term $o\bigl(\v_1(\Om_t)
\bigr)$ 
need not be small in comparison with $\pl(\Om_t)$. If $\v_1(\Om_t)
\ll \pl(\Om_t)$ holds as well, then the asymptotic formula simplifies
to
\begin{equation}\label{asfa}
\ct(\Om_t) \sim \Dtfct \pl(\Om_t) \quad (t\rightarrow\infty)\,.
\end{equation}

In this section we shall be content to discuss a number of families
for which the asymptotic formula \eqref{asf} holds, showing the
existence of automorphic forms with eigenvalue (or spectral)
parameters lying in such regions $\Om_t$, as $t$ gets large.

\subsubsection{Small rectangle in the real quadratic
case}\label{sss-rectquad} Before stating more general results, we
consider the case that $d=2$, and apply some of the more general
statements first for this situation.

Let $[\al,\bt]\subset [1/4,\infty)$ and consider for $t\geq \frac 54$:
\begin{equation}
\Om_t = [\al,\bt] \times [t,t+\sqrt t]\,.
\end{equation}
Theorem~\ref{thm-B-H-p} implies
\begin{equation}
\ct(\Om_t) \sim \frac{\sqrt{D_F}}{2\pi^2}\, \int_\al^\bt \tanh\pi
\sqrt{\ld-\txtfrac14}\, d\ld\;\; t^{1/2}\qquad(t\rightarrow\infty)\,.
\end{equation}
In particular, there are infinitely many $\varpi$ with $\ld_
{\varpi,1}\in [\al,\bt]$. These $\varpi$ have a second component of
unitary principal series type. A similar result holds with $\varpi_2$
of discrete series type:
\begin{align}
\Om_b &= [\al,\bt] \times \left\{ \txtfrac{b-1}2\right\}
\qquad\text { with }b>1\,,\; b \cong \x_2\bmod 2\,,\\
\nonumber \ct(\Om_b) &\sim \frac{\sqrt{D_F}}{2\pi^2} \, \int_\al^\bt
\tanh\pi \sqrt{\ld-\txtfrac14}\, d\ld\;\cdot\; b
\qquad(b\rightarrow\infty)\,.
\end{align}

On the other hand if $[\al,\bt] \subset \left[\lbs,\txtfrac14\right)$
then
\begin{align}
\text{for } \Om_t&=[\al,\bt]\times [t,\sqrt t] \qquad\text{with }t\geq
\txtfrac 54\,,\\
\nonumber
\ct(\Om_t) &= o\bigl(t^{1/2}\bigr)\qquad(t\rightarrow\infty)\,;\\
\text{and for } \Om_b &= [\al,\bt] \times \left\{
\txtfrac{b-1}2\right\} \qquad\text{ with }b>1\,,\; b \cong \x_2\bmod
2\,,\\
\nonumber \ct(\Om_b) &= o\bigl(b\bigr)\qquad(b\rightarrow\infty)\,.
\end{align}
This does not exclude the presence of $\varpi_1$ of complementary
series type, but gives an upper bound for their weighted density.

These results also hold with the role of $\varpi_1$ and $\varpi_2$
interchanged.
\medskip

One may also consider families of rectangles for which both factors
vary:
\begin{equation}
\Om_t = [t,t+\sqrt t]\times [ct,ct+\sqrt t]\,,
\end{equation}
with $t\geq \frac54$ and $c\geq 1$. Then
\begin{equation}
\ct(\Om_t) \sim \frac{\sqrt{D_F}}{2\pi^2} c^{1/2}
t\qquad(t\rightarrow\infty)\,.
\end{equation}

\subsubsection{Floating boxes} We consider for general $F$ a small
hypercube of fixed size floating to infinity in the region $\left(
i[1,\infty)\right)^d$ of the $\nu$-plane.
\begin{prop}\label{prop-hypercube}Let $\tilde\Om_t = \prod_j
i[a_j(t),a_j(t)+\s]$ with $\s>0$, 
$a_j(t)\geq 1$ for all $j$ and~$t$, and where
$\lim_{t\rightarrow\infty}a_j(t)=\infty$ for at least one $j$. Then
\[ \nct(\tilde\Om_t) \sim \Dtfct \npl(\tilde\Om_t)
\qquad(t\rightarrow\infty)\,. \]
\end{prop}
Note that $\npl(\tilde\Om_t) \rightarrow\infty$. We have
$\npl(\tilde\Om_t)\sim \s^d \prod_j a_j(t)$ if
$a_j(t)\rightarrow\infty$ for all~$j$. Proposition~\ref{box-restr}
implies that the size $\s$ may even slowly decrease with~$t$,
provided
\[ \s(t) \geq \g \biggl( \sum_j \log a_j(t) \biggr)^{-\al}\quad\text{
for any } \al\in (0,\frac12),\, \g>0\,. \]

We conclude that there are spectral parameters $\nu_\varpi$ in such
hypercube if they are sufficiently far away from the origin. If we
reformulate in terms of $\ld$ we get boxes $\Om_t$ in $\ld$-space
with increasing size.

\subsubsection{Remark} On the other hand, for hypercubes in
$\ld$-space with fixed size our method does not give an asymptotic
formula. In fact, the $\ld_\varpi$ may leave space for a small
hypercube moving around in $\ld$-space avoiding all of
the~$\ld_\varpi$. This can occur if the $\bigl|c^r(\varpi)\bigr|^2$
are not often very small.

\subsubsection{Discrete series} For $\varpi$ for which all factors are
discrete series representations, we do not need boxes, but can work
with single points.
\begin{prop}\label{prop-holo}Let $\pp \in \prod_j \left(
\frac{\x_j+1}2+\NN_0\right)$ and let $\Om = \Om_\pp= \{\pp\}$. Then,
if at least one coordinate of~$\pp$ tends to infinity, we have
\[ \nct\left( \{\pp\} \right) \sim \frac{2\sqrt{|D_F|}}{\pi^d} \prod_j
\pp_j\,.\]
\end{prop}

\subsubsection{Combinations} A combination is possible. At some places
we take a fixed interval $[A_j,B_j]$ on the $\ld$-line. We have to
require that the endpoints are not equal to a discrete series
eigenvalue:
\begin{equation}\label{bd-cond}
A_j, B_j \not\in \left\{ \txtfrac b2\left( 1-\txtfrac b2\right)\;:\;
b>1,\, b\equiv \x_j\bmod 2\right\}\,.
\end{equation}
\begin{thm}\label{thm-B-H-p}Let $E\sqcup Q_+ \sqcup Q_-$ be a
partition of the infinite places of~$F$. Suppose that $B_E =
\prod_{j\in E} [A_j,B_j]$ satisfies condition~\eqref{bd-cond}. Let
\[ C_t = \prod_{j\in Q_+} i\left[ a_j(t),a_j(t)+\s\right] \]
with $a_j(t)\geq 1$ for any $j\in Q_+$ and all~$t$, with $\dt>0$
fixed. Let $\pp \in \prod_{j\in Q_-} \left(
\frac{\x_j+1}2+\NN_0\right)$. Then, provided that
$a_j(t)\rightarrow\infty$ for at least one $j\in Q_+$ or
$\pp_j\rightarrow\infty$ one $j\in Q_-$, we have
\begin{align}
\nct \left( \tilde B_E \times C_t \times \{\pp\} \right) &= \Dtfct
\npl \left( \tilde B_E \times C_t \times \{\pp\} \right)\\
&\qquad\hbox{}
 + o\left( \nv_1(\tilde B_E \times C_t \times \{\pp\} )\right) \,.
\end{align}
\end{thm}
In Theorem~\ref{thm-B-H-p} some factors of $B_E$ may be in the region
$\left[\lbs, \frac14\right)$, giving $\npl(\tilde B_E \times
C_t \times \{\pp\})=0$. So here we need $\nv_1$ in the asymptotic
formula.

Again, the constant size $\s$ in Theorem~\ref{thm-B-H-p} can be
replaced by
\begin{equation} \s(t) = \g \biggl( \sum_{j\in Q_+}\log a_j(t) +
\sum_{j\in Q_-} \log(2\pp_j) \biggr)^{-\al}\qquad
(0<\al<\frac12\,,\; \g>0)\,.
\end{equation}

\subsubsection{Floating spheres} The following is analogous to an
unweighted distribution result, based on the Selberg trace formula,
given by Huntley and Tepper, \cite{HT92}:
\begin{prop}\label{prop-spheres}Let $B(\mm,r)$ be a sphere in $\RR^d$
with center $\mm$ and radius~$r$. Let $\mm_j \geq r+1$ for all~$j$,
and suppose that $\mm_j$ tends to infinity for all~$j$
(not necessarily with the same speed). Then
\[ \nct\left( iB(\mm,r)\right) \sim 4\sqrt{|D_F|} v_d \left( \txtfrac
r\pi\right)^d \prod_j \mm_j \,, \] where $v_d$ denotes the volume of
the unit sphere in~$\RR^d$.
\end{prop}
This asymptotic formula holds even if $r=r(\mm)$ is allowed to go down
to zero not faster than $\left( \log\prod_{j\in Q_+}\mm_j
\right)^{-\al}$, with $\al<\frac12$.

\subsubsection{Slanted strips in the quadratic case}
\begin{prop}\label{prop-slant}Let $d=2$, and put
\[ \tilde \Om_t = \left\{ (\nu_1,\nu_2) \in \left(i[1,\infty)
\right)^2\;:\; t \leq |\nu_1|\leq 2t,\, a|\nu_1|+b \leq |\nu_2| \leq
a|\nu_1|+c\right\}\,, \] with $a>0$, $c>b$ fixed and $t$ large. Then
\[ \nct(\tilde \Om_t) \sim \frac{14}{3\pi^2} \sqrt{|D_F|}
\;a(c-b)t^3\qquad(t\rightarrow\infty)\,. \]
\end{prop}
Proposition~\ref{prop-slant} shows that we see infinitely many points
$\nu_\varpi$ in a slanted direction. This slanted strip becomes a
sector in $\ld$-space. It opens up with a speed of the order
$|\ld_1|^{1/2}$.

\subsubsection{Density} Theorem \ref{thm-B-H-p} directly implies the
density results in Propositions 3.7 and 3.8 in~\cite{BMP3a}, without
the restriction $\ch=1$ and $\x=0$. The idea is that we take
$|E|=d-1$ in Theorem~\ref{thm-B-H-p}, and use $B_E$ to restrict
$(\ld_{\varpi,j})_{j\in E}$ to a tiny set with positive Plancherel
measure. The remaining coordinate of $\ld_\varpi$ is allowed to range
through a set in $\RR$ with increasing Plancherel measure. This shows
that there are infinitely many $\varpi$ with
$(\ld_{\varpi,j})_{j\in E} \in B_E$.

If one of the factors $[A_j,B_j]$ of $B_E$ is contained in the
interval $[\lbs,\frac14)$ of exceptional eigenvalues, the Plancherel
measure of $B_E \times C_t$ is zero. The asymptotic formula cannot
give the absence of $\varpi$ with $\ld_{\varpi,j} \in [A_j,B_j]$, but
only that the density is small in comparison with
$\nv_1(B_E\times C_t)$. We note that we project here along a
coordinate axis in the natural product structure of $\RR^d$ as the
product of the archimedean completions of~$F$.

\subsubsection{Weighted Weyl laws} In this subsection we list some
consequences of the asymptotic formula in the $\ld$-parameter. We
have that the asymptotic formula holds for $\Om_t = [-t,t]^d$. This
implies
\begin{prop}\label{prop-Weyl1}
\[ \sum_{\varpi,\, |\ld_{\varpi,j}|\leq t \text{ for all }j} \left|
c^r(\varpi) \right|^2 \sim \frac{2\sqrt{|D_F|}}{\pi^d}
t^d\qquad(t\rightarrow\infty)\,.\]
\end{prop}
This result confirms that there are infinitely many cuspidal~$\varpi$
for each choice of $\ch$ and $\x$ satisfying~\eqref{ch-x-comp}. With
the normalization of the~$c^r(\varpi)$ that we have chosen, the
density does not depend on the order of the Fourier coefficients that
we use.

\begin{prop}\label{prop-Weyl2}Let $Q_+ \sqcup Q_-$ a partition of the
real places of $F$. Then
\[ \sum_{\begin{array}{c}\varpi,\, \sum_j|\ld_{\varpi,j}|\leq t\\
\ld_{\varpi,j}\leq 0,\, j\in Q_-\\
\ld_{\varpi,j}>0\,,j\in Q_+\end{array}} \left| c^r(\varpi)
\right|^2 \sim \frac{2\sqrt{|D_F|}}{d!\, (2\pi)^d} t^d \]
\end{prop}
This variant is Corollary 3.4 in \cite{BMP3a}. There we considered
only the trivial character of $\G_0(I)$ and the trivial central
character. Proposition~\ref{prop-simp} implies that all results
in~\cite{BMP3a} extend to the more general context in this paper.

\subsubsection{Sectors in the quadratic case}\label{sss-sectquad}
\begin{prop}\label{prop-sector}Let $d=2$, and fix $0<p<q$,
$\al>\frac12$. For $t\geq \frac 54(1+\nobreak p^{-1})$ put
\[ \Om_t = \biggl\{ (\ld_1,\ld_2)\in [0,\infty)^2\;:\; t \leq
\ld_1\leq t+t^\al\,, p\ld_1\leq \ld_2\leq q\ld_1\biggr\}\,.\] Then
\[ \ct (\tilde \Om_ t) \sim \frac{q-p}4 t^{1+\al}
\qquad(t\rightarrow\infty)\,.\]
\end{prop}

\section{Sum formula} The basis of the result in this paper is the
Kuznetsov type sum formula in Theorem 3.21 of~\cite{BM6+}, which we
recall in \S\ref{sect-stat-sf}. In \S\ref{sect-appl-sf} we apply it
with test functions adapted to the present purpose.

\subsection{Statement of the sum formula}\label{sect-stat-sf}
The sum formula as stated in \cite{BM6+} depends on two non-zero
Fourier term orders $r,r'\in \Ocal'\setminus\{0\}$. For the end
results of this paper it suffices to take $r'=r$. In a later paper we
intend to work with different Fourier term orders and take Hecke
operators into account. Then we'll need to consider $r\neq r'$ as
well.

The sum formula states an equality with four terms, all depending on a
given test function. We first state the sum formula, and will next
recall the description of the ingredients.

\begin{thm}{\rm Spectral sum formula. }\label{thm-sf}For any test
function $\ph \in T_\x(\tau,a)$ the sums and integrals $\nctr
(\ph)$, $\ectr(\ph)$, $\npl(\ph)$ and $\K^r_\ch \left( \B^\sg_\x
\ph\right)$ converge absolutely, and
\begin{equation}
\label{sf} \nctr (\ph) + \ectr(\ph) = \Dtfct \npl(\ph) + \K^r_\ch
\left( \B^\sg_\x \ph\right).
\end{equation}\end{thm}

We work with a fixed character $\ch$ of $\G$ and a compatible central
character given by $\x\in \{0,1\}^d$; so condition \eqref{ch-x-comp}
is satisfied. Fixed is $r\in \Ocal'\setminus\{0\}$. Fixed are also
the parameters $\tau \in \left( \frac14,\frac12\right)$ and $a>2$,
which determine the space of test functions.

\subsubsection{Test functions}\label{sect-tf}The space $T_\x(\tau,a)$
of test functions consists of functions with the following product
structure:
\begin{equation}\label{phpr}
\ph(\nu_1,\ldots,\nu_d) = \prod_j \ph_j(\nu_j),
\end{equation}
where each $\ph_j$ is a function on a set
\begin{equation}
\left\{ z\in \CC\;:\; |\re z|\leq \tau\right\} \cup \left\{
   \txtfrac{b-1}2\;:\; b\equiv \x_j\bmod 2,\, b\geq2\right\}
\end{equation}
satisfying the following conditions:
\begin{enumerate}
\item[(T1)] $\ph_j$ is holomorphic on $|\re z|\leq \tau$;
\item[(T2)] $\ph_j(z) \ll \left(1+|z|\right)^{-a}$ on the domain of
$\ph_j$,
\item[(T3)] $\ph_j(-z) = \ph_j(z)$ on the strip $|\re z|\leq \tau$.
\end{enumerate}

\subsubsection{Spectral side}In the left hand side of \eqref{sf} are
two terms connected to the spectral decomposition of
$L^2_\x(\G_0(I)\backslash G,\ch)$. The first term $\nctr(\ph)$ is the
sum defined in~\eqref{tNrfdef}. Its convergence already implies the
existence of infinitely many cuspidal automorphic representations in
$L^{2,\mathrm{cusp}}_\x(\G_0(I)\backslash G,\ch)$.

The orthogonal complement in $L^2_\x(\G_0(I)\backslash G,\ch)$ of the
cuspidal subspace gives rise to the term $\ectr(\ph) $. Since $r$ is
non-zero, the constant functions, in the case $\ch=1$, $\x=0$, do not
contribute to the sum formula. We have:
\begin{align}\label{ectrdef}
&\ectr(\ph) \displaybreak[0]
\\
\nonumber &\quad= \sum_{\k\in \Pcal_\ch} c_\k \sum_{\mu\in
\Ld_{\k,\ch}} \int_{-\infty}^\infty \left| D^r_\x(\k,\ch;it,i\mu)
\right| \ph(it+i\mu)\, dt.
\end{align}
$\Pcal_\ch$ is a set of representatives of cuspidal orbits suitable
for the character $\ch$. For each $\k$, there is a lattice
$\Ld_{\k,\ch}$in the hyperspace $\sum_j x_j=0$ in $\RR^d$. In
$it+i\mu$, the real number $t$ is identified to
$(t,t,\ldots,t)\in \RR^d$. The positive constants $c_\k$ come from
the spectral formula for the continuous spectrum. The
$D^r_\x(\k,\ch;it,i\mu)$ are normalized Fourier coefficients of
Eisenstein series. See (2.31) of~\cite{BM6+}. There is a positive
real number $q$ such that
\begin{equation}
\label{Dest} D^r_\x(\k,\ch;it,i\mu) \ll_{F,I,r}
\begin{cases}
\left(\log(2+|t|)\right)^q+ \left(\log
\max_j(|\mu_j|+1)\right)^q&\text{ if }\mu\neq0,\\
\left(\log(2+|t|)\right)^q&\text{ if }\mu=0;
\end{cases}
\end{equation}
This is discussed in Proposition 5.2 in \cite{BMP3a} and \S4.2 in
\cite{BM6+}, with $q=7$.

Actually, it is conceivable that for some fields $F$, some ideals $I$
and some cusps $\k$, there might be $\mu \in \Ld_{\k,\ch}$ with
$\mu\neq 0$ and $\max_j |\mu_j|<1$. This is not intended in
\cite{BMP3a}. The logarithms arise in (71) and (72) of \cite{BMP3a}.
Checking the reasoning there, we see that if the bounds for
$Z(s,\ld,\tau)$ are negative, they can be replaced by $0$. Thus, we
can replace $\left(\log \max_j \bigl(|\mu_j|+1\bigr)\right)^7$ in
\eqref{Dest} by $\left(\log \bigl(
\max_j( |\mu_j|+2)\bigr)\right)^7$.

There is another reformulation. Put $x_i = t + \mu_i$. Then $t =
\frac1d \sum_j x_j$ and $\mu_i = x_i - \frac1d \sum_j x_j$. This
implies that we have $|t| \ll \sum_j|x_j|$ and $\|\mu\| \ll
\sum_j|x_j|$. Thus we arrive at
\begin{equation}\label{Dest1}
D^r_\x(\k,\ch;it,i\mu) \ll_{F,I,r} \biggl(\log \bigl( 2 +
\sum_j|t+\mu_j| \bigr) \biggr)^7.
\end{equation}

\subsubsection{Delta term}The right hand side in \eqref{sf} arises
from geometrical considerations. For the purpose of this paper the
first term is the principal one.

\subsubsection{Bessel transformation}The last term in \eqref{sf}
contains a \emph{Bessel transform} of the test function. It depends
on $r\in \Ocal'\setminus\{0\}$ by $\sg = (\sign r_j) _j \in
\{1,-1\}^d$. Here we view elements of $F$ as elements of $\RR^d
=\prod_j F_j$. The Bessel transformation has a product form:
\begin{align}\label{bt}
\B^\sg_\x\ph(t) &= \prod_j \B^{\sg_j}_{\x_j}\ph_j(t_j),\\
\nonumber \B_0^\eta\ph(t) &\isdef -\txtfrac i2 \int_{\re\nu=0}
\ph(\nu) \frac{J_{2\nu}(|t|)-J_{-2\nu}(|t|)}{\cos\p\nu}\, \nu\, d\nu
\displaybreak[0]
\\
\nonumber &\qquad\hbox{} + \sum_{b\geq 2,\, b\equiv0\, \bmod2}
(-1)^{b/2}(b-1) \ph\left(\txtfrac{b-1}2\right) J_{b-1}(|t|),
\displaybreak[0]
\\
\nonumber \B_1^\eta\ph(t) &\isdef -\txtfrac{\eta}2 \sign(t)
\int_{\re\nu=0} \ph(\nu)
\frac{J_{2\nu}(|t|)+J_{-2\nu}(|t|)}{\sin\p\nu} \,\nu\, d\nu
\displaybreak[0]
\\
\nonumber &\qquad\hbox{} - i\eta \sign(t) \sum_{b\geq3, \, b\equiv1\,
\bmod2} (-1)^{(b-1)/2} (b-1)
\ph\left(\txtfrac{b-1}2\right) J_{b-1}(|t|).
\end{align}

These Bessel transforms converge absolutely for any test function, and
provide us with functions $f=\B^\sg_\x\ph$ on~$\left(
\RR^\ast\right)^d$ that satisfy
\begin{equation}\label{Bfest}
f(t) \ll \prod_j \min \left( |t_j|^{2\tau},1\right).
\end{equation}

\subsubsection{Sum of Kloosterman sums} The Kloosterman sums for the
present situation are
\begin{equation}\label{klsdef}
\kls\ch(r',r;c) = \summ\ast{a\, \bmod(c)} \ch(a) e^{2\pi
 i\mathrm{Tr}_{F/\QQ}((ra+r'\tilde a)/c)},
\end{equation}
with $r,r'\in \Ocal'$ and $c\in I \setminus\{0\}$. This is well
defined if $\ch$ is a character of $\left( \Ocal/(c)
\right)^\ast$, in particular for the character $\ch$ of $\left(
\Ocal/I \right)^\ast$ if $c\in I$. For each $a\in \Ocal$ that is
invertible modulo $(c)$, an element $\tilde a$ is picked such that
$a\tilde a\equiv 1\bmod (c)$.

For functions $f$ satisfying \eqref{Bfest} the sum
\begin{equation}\label{Kdef}
\K^r_\ch (f) \isdef \sum_{c\in I\setminus\{0\}} |N(c)|^{-1}
\kls\ch(r,r;c) f\left(\txtfrac{4\p|r|}c\right)
\end{equation}
converges absolutely. This holds in particular for $f=\B^\sg_\x \ph$.
By $\frac {4\p|r|}c$ we mean the element of $\left(
\RR^\ast\right)^d$ given by the $d$ embeddings $F\rightarrow\RR$ of
$r$ and $c$.

A trivial estimate of the Kloosterman sums is
$\left|\kls\ch(r',r;c)\right| \leq |N(c)|$. As discussed in \S2.4 of
\cite{BM6+}, there is a Weil type estimate. If the ideal $(c)$ has a
representation $(c) = \prod_\prm \prm^{v_\prm(c)}$ in prime ideals,
then
\begin{equation}\label{We}
\left|\kls\ch(r',r;c)\right| \ll_\dt \left|N(rr')\right|^{1/2} \;
\prod_{\prm \dividesnot I} \left( N\prm\right)^{v_\prm(c)/2+\dt} \;
\prod_{\prm\divides I} \left( N\prm\right)^{v_\prm(c)+\dt}
\end{equation}
for each $\dt>0$. This is not the best possible estimate, however it
is reasonably simple and will do for our purpose.

\subsection{Application of the sum formula}\label{sect-appl-sf}
In view of the term $\nct(\ph)$ in the sum formula, we want to choose
the test function $\ph$ such that it approximates the characteristic
function of a compact set $\tilde \Om$ in the space $Y_\x$, in which
the spectral parameters $\nu_\varpi $ take their values. In this
paper we first choose a test function approximating the delta
distribution at $q\in \prod_{j\in Q} Y_{\x_j}$ for a non-empty
subset~$Q$ of real places. At the other archimedean places we leave
$\ph_j$ free for the moment in the space of local test functions. The
local factor $\ph_j$ for $j\in Q$ such that $q_j =\frac{b-1}2$,
$b\cong \x_j\bmod 2$, $b\geq 2$ can be chosen such that practically
$\ph_j$ is the delta distribution at $\ph_j$. For
$q_j \in \left[0, \nubs\right)\cup i[0,\infty)$ the choice is more
delicate. It does not suffice that $\ph_j$ approximates the delta
distribution at~$q_j$. The terms $\ectr(\ph)$ and
$\K^r(\B^\sg_\x \ph)$ should have good estimates. Under the
additional assumption that $q_j \not\in
(0,\nubs]\cup i\, [0,1]$, the choice that worked best is a sharp
Gaussian function, similar to, but slightly simpler, than the test
function used in, for instance, \cite{IvJu} and~\cite{JuMo}. For
$q_j\in (0,\nubs]\cup i\,[0,1)$ we have not found a test function
that works well.

\begin{lem}\label{lem-sf-q}
Let $\{1,\ldots,d\} = E \sqcup Q_+ \sqcup Q_-$ with $Q\isdef Q_+ \cup
Q_- \neq \emptyset$. For $q \in \left( \RR \setminus \left(
-\frac12,\frac12\right) \right)^{Q_-} \times \left( i\RR \setminus
i(-1,1) \right)^{Q_+}$ put $\ph(q,\nu) = \prod_ j \ph(q_j,\nu_j)$
where $\ph_j(q,\cdot)$ is an arbitrary local test function satisfying
the conditions in \S\ref{sect-tf} if $j\in E$, and where for
$j\in Q$:
\begin{equation}
\label{phi-q-def}
\renewcommand\arraystretch{1.5}
\begin{array}{|c|l|}\hline
   &\multicolumn{1}{|c|}{\ph_j(q,\nu)}
\\ \hline
j\in Q_+& \begin{cases}\sqrt{\frac {U}\p} \left( e^{U(q-\nu)^2} +
e^{U(q+\nu)^2} \right)&\text{ if }|\re\nu|\leq \tau\\
0&\text{ elsewhere}
\end{cases}
\\ \hline
j\in Q_- & \begin{cases}1&\text{ if }\nu =q\text{ or }-q\\
0&\text{ elsewhere}
\end{cases}
\\ \hline
\end{array}
\end{equation}
Then there are constants $t_0>0$, $ \r\in
(1-\nobreak\tau,1)$ such that for $U\geq 1$ and $A>2$
\begin{align}\label{q-form}
\nct\bigl( \ph(q&,\cdot) \bigr) = \Dtfct 
\npl\left( \ph(q,\cdot)
\right)\\
\nonumber &\quad\hbox{} + \oh_{F,I,r,t_0,t_1,A}\biggl( N_E(\ph_E)
e^{t_0 U |Q_+|} \prod_{j\in Q_+} |q_j|^\r \prod_{j\in Q_-} |q_j|^{-A}
\biggr),
\end{align}
where $\ph_E = \bigotimes _{j\in E}\ph_j$ and
\begin{align}\label{NE-def}
N_E(\ph_E) &= \prod_{j\in E} N_j(\ph_j),\\
\nonumber N_j(\ph_j)&= \sup_{\nu,\, 0\leq \re\nu \leq \tau}
|\ph(\nu)|(1+|\nu|)^a + \sum_{b\equiv \x_j(2),\, b\geq 2} b^a\;
\left|\ph_j\left(\txtfrac{b-2}2\right)\right|.
\end{align}
\end{lem}
Note that $\ph_j$ chosen in \eqref{phi-q-def} is a local test function
for any choice of the parameters $a>2$ and $\tau\in \left(
\frac14,\frac12\right)$.

By Theorem~\ref{thm-sf} it suffices to estimate $\ectr(\ph(q,\cdot))$
and $\K^r_\ch(\B^\sg_\x\ph(q,\cdot))$ by the error term
in~\eqref{q-form}. This we carry out in the remainder of this
subsection.

\subsubsection{Bessel transforms} For the factors $j\in E$ we cannot
do better than apply Lemma 3.12 in \cite{BM6+}. This gives
\begin{equation}
\B^{\sg_j,\sg_j'}_{\x_j} \ph_j(t) \ll_E \min \left(
|t|^{2\tau},1\right).
\end{equation}
The subscript $E$ in $\ll_E$, $\oh_E$ and $o_E$ indicates not only
dependence on the choice of the set $E$, but also on the choice of
the test function $\ph_E\isdef \bigotimes_{j\in E}\ph_j$. Here and in
the sequel, this dependence goes via the factor $N_E(\ph_E)$ in
\eqref{NE-def}.

For $j\in Q_-$ we first consider $y \leq 2\sqrt n$, with $n\in \NN$.
Then $\left|J_n(y)\right| \leq \frac{(y/2)^n
e^{\frac14y^2}}{n!} \ll y^{2\tau} \frac{ n^{\frac12 n-\tau}
e^n}{(n+1)^{n+\frac12} e^{-n}} \ll y^{2\tau} \left( \frac{n e^4
}{(n+1)^2} \right)^{\frac12 n} n^{-\frac12-\tau} \ll_b \frac
{y^{2\tau}}{n^b} $ for each $b>0$. For $y \geq \sqrt n$:
$\left|J_n(y)\right|\leq 1$. Hence, for $j\in Q_-$ and $\pm q_j \in
\frac{\x_j-1}2 +\NN$:
\begin{equation}\label{BestQ-}
\B^{\sg_j}_{\x_j}\ph_j(q;t) \ll_A \min \left(
|q|^{-A_1}|t|^{2\tau},|q|\right)\qquad\text{ for each }A_1>0.
\end{equation}
If $|q_j|\not\in \frac{\x_j-1}2 +\NN$, then
$\B^{\sg_j,\sg_j}_{\x_j}\ph_j(t) =0$.

The case $j\in Q_+$ takes more work. The function $\ph_j(q,\nu)$ is
non-zero only for $|\re\nu|\leq \tau$. We need an estimate like
\eqref{BestQ-}, in which the dependence on $q$ is explicit.

We may use (3.64) in \cite{BM6+}:
\begin{equation}\label{Btauint}
\B^{\sg_j}_{\x_j} \ph_j(q;t) = -i \left( i\sg_j \sign t \right)^{\x_j}
\int_{\re\nu=\tau} \ph_j(q,\nu) \frac{\nu J_{2\nu}(|t|)}{\cos\p\left(
\nu-\frac{\x_j}2\right)}\, d\nu.
\end{equation}

We proceed as in \S4.2 of \cite{BMP3a}, and apply the integral
representation in (41) of {\sl loc.\ cit.}, with $\al=\re\nu=\tau$
and $\g\in \left( \tau,\frac12\right)$. As in \cite{BMP3a} p.~700,
this leads to the estimate for all $|t|>0$:
\begin{equation}\label{Jai}
J_{2\nu}(|t|) \ll |t|^{2\tau} e^{\p|\im\nu|}
\left(1+|\im\nu|\right)^{\frac12-\g-\tau}.
\end{equation}

We obtain for $q\in i\RR$:
\begin{align}
\B^{\sg_j}_{\x_j}&\ph_j(q;t) \ll U^{1/2} \int_{-\infty}^\infty
\sum_\pm e^{U\re(ix+\tau\pm q)^2} |t|^{2\tau} \left|\tau+ix\right|
\displaybreak[0]
\\
\nonumber &\qquad\qquad\hbox{} \cdot
\left( 1+|x|\right)^{\frac12-\tau-\g}\, dx\displaybreak[0]\\
\nonumber &\ll |t|^{2\tau} U^{1/2} \int_{-\infty}^\infty
e^{U\tau^2-U(x-|q|)^2} \left( 1+|x|\right)^{\frac 32-\g-\tau}\, dx
\displaybreak[0]\\
\nonumber &\ll |t|^{2\tau} e^{U\tau^2} \int_0^\infty e^{-x^2} \left(
1+|q|+\txtfrac x{\sqrt U} \right) ^{\frac32-\g-\tau}\, dx
\displaybreak[0]\\
\nonumber &\ll |t|^{2\tau} e^{U\tau^2} \int_0^\infty e^{-x^2} \biggl(
\left( 1+|q|\right)^{\frac32-\g-\tau} + \bigl(x U^{-1/2}
\bigr)^{\frac 32-\g-\tau} \biggr)\, dx
\displaybreak[0]\\
\nonumber &\ll |t|^{2\tau} e^{U\tau^2} \biggl( (1+|q|)^{\frac
32-\g-\tau} + U^{-\frac34+\frac\g2+\frac
\tau2}\biggr)\displaybreak[0]\\
\nonumber &\ll |t|^{2\tau} e^{U\tau^2} (1+|q|)^{\r_1},
\end{align}
where $\r_1 = \frac 32-\g-\tau$; so $\r_1\in \left(
1-\tau,\frac32-2\tau\right) \subset\left( \frac12,1\right)$.

The factor $e^{U\tau^2}$needs attention since it becomes large for
large values of~$U$. We carry along this factor, and will compensate
for it later.
\medskip

These estimates are good for small values of $|t|$. For large $|t|$,
we use
\begin{equation}\label{Je1}
J_{2\nu}(|t|) \ll e^{\p|\im\nu|}\qquad \text{ for }\re\nu=0;
\end{equation}
see (5.44) in \cite{BM6+}, with $\s=0$.
\begin{align*}
\B^{\sg_j}_{\x_j}\ph_j(q,t) &\ll U ^{1/2} \int_{\re\nu=0} \left| e^{U
\left( \nu+q\right)^2} + e^{U \left( \nu-q\right)^2} \right| |\nu|\,
|d\nu| \displaybreak[0]
\\
\nonumber &\ll \int_{-\infty}^\infty e^{-x^2} \left| x U ^{-1/2}+
|q|\right|\, dx \ll 1+|q| \ll |q|.
\end{align*}

Thus we have
\begin{equation}
\B^\sg_\x\ph(q;t) \ll_{A_1} \prod_j \left( a_j
|t_j|^{2\tau},b_j\right),
\end{equation}
with
\begin{equation}\label{ab-tab}
\renewcommand\arraystretch{1.4}
\begin{array}{|c|c|c|}\hline
a_j&b_j&\\ \hline N_j(\ph_j)&N_j(\ph_j)&j\in E\\ \hline
|q|^{-A_1}&|q|&j\in Q_-\\ \hline e^{\frac12\tau^2 U} |q|^{\r_1} & |q|&
j\in Q_+\\ \hline
\end{array}
\end{equation}

\subsubsection{Kloosterman term} We estimate the sum of Kloosterman
sums by the sum of the absolute values of the terms:
\begin{equation}\label{sKls-e}
\K^{r,r}_\ch\left( \B^{\sg,\sg}_\x\ph(q;\cdot)\right) \ll \sum_{c\in
I\setminus\{0\}} \frac{\left|\kls\ch(r,r;c)\right|}{|N(c)|} \prod_j
\min \biggl( a_j \biggl(
\txtfrac{4\p|r_jr_j'|^{1/2}}{|c_j|}\biggr)^{2\tau},b_j \biggr),
\end{equation}
with $a_j$ and $b_j$ as in~\eqref{ab-tab}.

For the Kloosterman sums, we use the Weil bound, as stated in
\eqref{We}. This bound depends only on the ideal $(c)$, so we
decompose the sum as
\begin{align*}
&\ll_{r,\dt} \summ'{(c)\subset I} \prod_{\prm\dividesnot I}
N\prm^{v_\prm(c)(\dt-1/2)} \prod_{\prm\divides I} N\prm^{v_\prm(c)\dt
}\\
&\qquad\hbox{} \cdot \sum_{\z\in \Ocal^\ast} \prod_j \min \left( p_j
|\z_j|^{-2\tau}, q_j \right)\,,\\
&\text{with }\quad p_j=a_j (4\pi|r_j|)^{2\tau} ,\quad q_j=b_j,\quad
(c) = \prod_{\prm\text{ prime}} \prm^{v_\prm(c)}.
\end{align*}
The prime denotes that the zero ideal is excluded. We can take $\dt>0$
as small as we want.

We apply Lemma 2.2 in \cite{BM6+} with $\al=2\tau$, $\bt=0$ and $y_j =
c_j^{-1}$. Thus, we estimate the sum over $\z\in \Ocal^\ast$ by:
\begin{align*}
&\ll \biggl( 1+ \bigl| \log|N(c)| + \frac1{2\tau} \log\frac{\mathbf
Q}{\mathbf P} \bigr|\biggr)^{d-1} \min \left( \mathbf P
|N(c)|^{-2\tau},\mathbf Q\right),
\displaybreak[0]\\
\mathbf P &= \prod_j p_j = (4\pi)^{2\tau d} |N(r)|^{2\tau}
e^{\frac12\tau^2 U|Q_+|} N_E(\ph_E) \prod_{j\in Q_+} |q_j|^{\r_1}
\prod_{j\in Q_-}|q_j|^{-A_1},
\displaybreak[0]\\
\mathbf Q &= \prod_j q_j = N_E(\ph_E)
\prod_{j\in Q} |q_j|,
\displaybreak[0]\\
\frac{\mathbf P}{\mathbf Q} &= (4\pi)^{2\tau d}e^{\frac12\tau^2 U
|Q_+|} |N(r)|^{2\tau} \prod_{j\in Q_+} |q_j|^{\r_1-1}\prod_{j\in Q_-}
|q_j|^{-A_1-1}.
\end{align*}

We have already a small quantity $\dt$. We employ it also for the
logarithms. We use that for $c\in \Ocal\setminus\{0\}$ the $|N(c)|$
stay away from zero:
\begin{align*}
\left|\log|N(c)|\vphantom{\log\frac{\mathbf Q}{\mathbf P}}
\right.&\left.+\txtfrac1{2\tau} \log\txtfrac{\mathbf Q}{\mathbf
P}\right| \leq \left| \log |N(c)|\right| + \txtfrac1{2\tau}
\bigl|\log \txtfrac{\mathbf Q}{\mathbf P} \bigr|\\
&\ll_\dt |N(c)|^{\dt/(d-1)} + \max \biggl( \bigl(\txtfrac{\mathbf
Q}{\mathbf P}\bigr)^{\dt/(d-1)} , \bigl(\txtfrac{\mathbf P}{\mathbf
Q}\bigr)^{\dt/(d-1)} \biggr)\,,
\displaybreak[0]
\\
\left|\log|N(c)|\vphantom{\log\txtfrac{\mathbf Q}{\mathbf P}}
\right.&\left.+\txtfrac1{2\tau} \log\txtfrac{\mathbf Q}{\mathbf
P}\right|^{d-1} \ll |N(c)|^\dt + \max \left( \left(\txtfrac{\mathbf
Q}{\mathbf P}\right)^\dt , \left(\txtfrac{\mathbf P}{\mathbf
Q}\right)^\dt \right),\\
\biggl(1+\bigl|\log|N(c)| & +\txtfrac1{2\tau} \log\txtfrac{\mathbf
Q}{\mathbf P}\bigr| \biggr)
\min \left( \mathbf P |N(c)|^{-2\tau}, \mathbf Q\right)\\
&\ll_{r} e^{\frac12\tau^2 U|Q_+|(1+\dt)} N_E(\ph_E)
|N(c)|^{-2\tau(1-\dt)}\\
&\qquad\hbox{} \cdot \prod_{j\in Q_+}|q_j|^{\r+(1-\r_1)\dt }
\prod_{j\in Q_-} |q_j|^{-A_1+(A_1-1)\dt }\,.
\end{align*}
We have assumed that the first factor in the minimum is the essential
one for our purpose.

Let us define: $t_0= \frac12\tau^2(1+\nobreak \dt)$,
$\tau_1=\tau(1+\nobreak \dt)$, $\r = \r_1+(1-\nobreak \r_1)\dt$, and
$A=A_1+(1-A_1)\dt$. We take $A_1=-3$, and $\dt>0$ sufficiently small
such that $\frac14<\tau_1<\frac12$, $1-\tau<\r<1$, $A>2$. We find the
following estimate for the sum of Kloosterman sums in~\eqref{sKls-e}:
\begin{align*}
&\ll_{r,\dt} e^{t_0 U|Q_+| } N_E(\ph_E)
\summ \prime{(c)\subset I} \prod_{j\in Q_+}|q_j|^\r \\
&\qquad\hbox{} \cdot \prod_{j\in Q_-} |q_j|^{-A}
\prod_{\prm\dividesnot I} N\prm^{v_\prm(c)(\dt-1/2 -2\tau_1) }
\prod_{\prm \divides I} N\prm^{v_\prm(c)(\dt
-2\tau_1) } \\
&\ll e^{t_0 U|Q_+| } N_E(\ph_E) \prod_{j\in Q_+}|q_j|^\r \prod_{j\in
Q_-} |q_j|^{-A}\\
&\qquad\hbox{} \cdot \prod_{\prm\dividesnot I} \frac
1{1-N\prm^{\dt-1/2-2\tau_1}} \prod_{\prm\divides
I}\frac1{1-N\prm^{\dt-2\tau_1}}.
\end{align*}
Under the additional assumption on $\dt$ that $2\tau_1+\frac12-\dt>1$,
the product converges, and we have obtained:
\begin{equation*}
\K^{r,r}_\ch\left( \B^{\sg,\sg}_\x\ph(q;\cdot)\right)
\ll_{F,I,r,\dt} e^{t_0 U|Q_+| } N_E(\ph_E) \prod_{j\in Q_+}|q_j|^\r
\prod_{j\in Q_-} |q_j|^{-A},
\end{equation*}
with the size of the error term in~\eqref{q-form}.

This is the main place where the dependence on the ideal
$I\subset\Ocal$ determining $\G=\G_0(I) \subset \SL_2(\Ocal)$ enters
the estimates. We leave this dependence implicit.

\subsubsection{Eisenstein term}We still have to estimate
$\ectr(\ph(q,\cdot))$. The definition in \eqref{ectrdef} shows that
$\ectr(\ph(q,\cdot))=0$ if $Q_-=0$.

For $j\in E$, we have $\ph_j(\nu)\ll_E \left( 1+|\nu|^2\right)^{-a}$.
In view of \eqref{Dest1} it suffices to estimate
\begin{align*}
\sum_{\k\in \Pcal_\ch} & \sum_{\mu \in \Ld_{\k,\chi}}
\int_{-\infty}^\infty |N(r)|^{2\dt} N_E(\ph_E) l(q,t+\mu)\, dt,
\mbox{ with }
\displaybreak[0]\\
\nonumber l(q,x) &= \prod_{j\in E} \left( 1+x_j^2\right)^{\dt-a/2}
\displaybreak[0]\\
\nonumber &\qquad\hbox{} \cdot \prod_{j\in Q_+} U^{1/2} \left(
1+x_j^2\right)^\dt \left( e^{-U(x_j-|q_j|)^2} + e^{-U(x_j+|q_j|)^2}
\right).
\end{align*}
In \S5.3 of \cite{BMP3a}, we have replaced the sum over $\mu\in
\Ld_{\k,1}$ by an integral over the hyperplane $\sum_{j=1}^d x_j=0$.
Since in \cite{BMP3a} the quantity corresponding to $U$ went down to
zero, there was no problem there. Here $U$ may be large, and we have
to take a closer look at the relation between the sum and the
integral.

The integral gives a contribution, under the assumption $\frac
a2-\dt>1$:
\begin{align}\label{lint}
\int_{\RR^d} &l(q,x)\, dx \ll \prod_{j\in E} \int_{-\infty}^\infty
(1+x^2)^{\dt-a/2}\, dx \displaybreak[0]
\\ \nonumber&\qquad\hbox{} \cdot
\prod_{j\in Q_+} U^{1/2} \int_{-\infty}^\infty (1+x^2)^\dt
e^{-U(x-|q|)^2}\, dx
\displaybreak[0]\\
\nonumber &\ll \prod_{j\in E} 1\; \prod_{j\in Q_+}
\int_{-\infty}^\infty \left( 1+ |q_j|^2 + U^{-1} x^2\right)^\dt
e^{-x^2}\, dx
\displaybreak[0]\\
\nonumber &\ll \prod_{j\in Q_+} \left( 1+|q_j|\right)^{2\dt}
\left(1+\txtfrac 1U\right) \ll \prod_{j\in Q_+} |q_j|^{2\dt}.
\end{align}

The difference between the value at $\mu\in \Ld_{\k,\ch}$ and the
integral over $\mu+V$, where $V$ is a compact neighborhood of $0$,
produces an error estimated by the gradient of $l(q,\cdot)$.
\begin{align*}
\frac{\partial}{\partial x_j} l(q,x) &\ll_{a,\dt}
l(q,x)\cdot\begin{cases}
\frac{|x_j|}{1+x_j^2} &\text{ if }j\in E,\\
\left( \frac{|x_j|}{1+x_j^2} +U\left|x_j-|q_j|\right|\right)
&\text{ if }j\in Q_+.
\end{cases}
\end{align*}
The difference between the sum and the integral is estimated by
\begin{equation*}
\int_{\RR^d} \biggl( \sum_{m\in E} \oh(1) + \sum_{m\in Q_+} \left(
1+U\left|x_m-|q|_m\right| \right) \biggr) l(q,x)\, dx.
\end{equation*}
The terms with $m\in E$ can be estimated as in~\eqref{lint}. For a
term with $m\in Q_+$:
\begin{align*}
&\ll \int_{-\infty}^\infty \left( 1+U\left|x_m-|q|_m\right| \right)
U^{1/2}(1+x_m^2) e^{-U(x-|q_m|)^2}\, dx
   \\ \nonumber &\qquad\hbox{} \cdot
\prod_{j\in E} \oh(1) \prod_{j\in Q_+\setminus\{m\}} |q_j|^{2\dt}
\displaybreak[0]\\
\nonumber
&\ll \prod_{j\in Q_+\setminus\{m\}} |q_j|^{2\dt} \\
\nonumber &\qquad\hbox{} \cdot \int_{-\infty}^\infty
\left(1+U^{1/2}|x|\right) \left( 1+|q_m|^2 + U^{-1} x^2\right)^\dt
e^{-x^2}\, dx
\displaybreak[0]\\
\nonumber &\ll U^{1/2} \prod_{j\in Q_+}|q_j|^{2\dt}\,.
\end{align*}

Thus, we obtain:
\begin{equation}\label{Eisestq}
\ectr \left(\ph(q,\cdot)\right) \ll_{F,I,r,\dt} N_E(\ph_E) U^{1/2}
\prod_{j\in Q_+} |q_j|^{2\dt}\,.
\end{equation}
Since $\dt>0$ can be as small as we desire, this bound is easily
absorbed into the error term in~\eqref{q-form}.

\subsubsection{Delta term}
\begin{lem}\label{lem-npl-comp}
Let $\nplf_j$ as in \eqref{nplf-def}. For $E$, $Q_+$, $Q_-$ and
$\ph(q,\cdot)$ as in Lem\-ma~\ref{lem-sf-q}:
\begin{align}\label{npl-comp}
\npl_\x(\ph(q,\cdot)) &= \prod_{j\in E} \npl_{\x_j}(\ph_j)
\prod_{j\in Q} 2\nplf_j(q_j) \\
\nonumber &\qquad\hbox{} + \begin{cases} U^{-1/2}N_E(\ph_E)
|Q_+|\frac{\prod_{j\in Q}{|q_j|}}{\min_{j\in Q_+} |q_j|}&\text{ if
}Q_+\neq\emptyset,\\
0&\text{ if }Q_+=\emptyset.
\end{cases}
\end{align}
\end{lem}
\begin{proof}Since
\[ \npl_{\x_j}(\ph_j(q,\cdot)) =
\begin{cases}
2|q| &\text{ if }2q \equiv \x_j-1\bmod 2,\\
0&\text{ otherwise},
\end{cases}
\]
we have to consider the discrepancy between
$\npl_{\x_j}(\ph_j(q,\cdot))$ and $2\nplf_j(q_j)$ for $j\in Q_+$.

The function $t\mapsto \nplf_j(it)$ is even and smooth on $\RR$. If
$\x_j=0$, then $\nplf_j(0)=0$, and if $\x_j=1$, then
$\nplf_j(0)=\frac1\p$. We have $\plf_j(it) \sim |t|$ as
$|t|\rightarrow\infty$, and $\frac d{dt} \plf_j(it) = \oh(1)$ for
$t\in \RR$.
\begin{align}
\npl_{\x_j}&\left(\ph_j(q_j,\cdot)\right) - 2\nplf_j(q_j)
\displaybreak[0]\\
\nonumber &= 2\sqrt{\frac U\pi} \int_0^\infty \left( e^{-(x-|q_j|)^2}
+ e^{-(x+|q|)^2} \right)
\nplf_j(ix)\, dx - 2\nplf_j(q_j)\displaybreak[0]\\
\nonumber &= 2\sqrt{\frac U\p} \int_{-\infty}^\infty e^{-U(x-|q_j|)^2}
\nplf_j(x)\, dx - 2\nplf_j(q_) \displaybreak[0]
\\
\nonumber &\quad= \frac2{\sqrt\p} \int_{-\infty}^\infty e^{-x^2}
\left( \nplf_j \left( q_j+ \txtfrac {ix}{\sqrt U} \right) -
\nplf_j\left( q_j\right) \right)\, dx.
\end{align}
We write $ \nplf_j \left( q_j+ \frac {ix}{\sqrt U} \right) -
\nplf_j\left( q_j\right) = U^{-1/2} x\;\frac d{d\th}\nplf_j(i\th)
$ for $|x|\leq b$, with $b\geq 1$, and $\th$ between $|q_j|$ and
$|q_j|+\frac x{\sqrt U}$.
\begin{align*}\npl_{\x_j}&\left(\ph_j(q_j,\cdot)\right)
- 2\nplf_j(q_j)
\\
&\ll \int_{-b}^b e^{-x^2} \oh(1) U^{-1/2} |x|\, dx + \int_{|x|\geq b}
e^{-x^2} \oh\left( |q_j| + |x| U^{-1/2} \right)\, dx
\displaybreak[0]\\
&\ll U^{-1/2} + |q_j| \frac{e^{-b^2}}b + U^{-1/2} e^{-b^2}.
\end{align*}
Here we have used that for $b\in \RR$ and $l\geq 0$:
\begin{equation}\label{Gint}
\int_b^\infty |x|^l e^{-x^2}\, dx\ll_{l}
\begin{cases}
1&\text{ if } b\leq 1,\\
b^{l-1} e^{-b^2} &\text{ if }b\geq 1;
\end{cases}
\end{equation}
which can be checked by partial integration and induction.

We assume that $U\geq e^2$, and choose $b = b(q, U) = \sqrt{ \log|q_j|
+\frac12\log U}$, which satisfies $b\geq 1$. This gives
\begin{equation*}
\npl_{\x_j}\left(\ph_j(q_j,\cdot)\right) - 2\nplf_j(-iq_j) \ll
U^{-1/2}.
\end{equation*}

Furthermore, we have
\begin{alignat}2 \label{npl-est}
\nplf_j(q_j) &\ll |q_j|&&(j\in Q_+),\displaybreak[0]\\
\nonumber \npl_{\x_j}\left( \ph_j(q_j,\cdot)\right) &\ll
U^{-1/2}+|q_j| \ll |q_j|&\qquad&(j\in Q_+),
\displaybreak[0]\\
\nonumber \nplf_{\x_j}\left( \ph_j(q_j,\cdot)\right) &\ll |q_j|&&(j\in
Q_-),
\displaybreak[0]\\
\nonumber \npl_{\x_j}(\ph_j) &\ll N_j(\ph_j)&&(j\in E).
\end{alignat}

These local estimates imply that
\begin{align*}
\prod_{j\in Q_+} \npl_{\x_j}&(\ph_j(q_j,\cdot)) - \prod_{j\in Q_+}
2\nplf_j(q_j)
\ll \sum_{m\in Q_+} U^{-1/2} \prod_{j\in Q_+,\, j\neq m} |q_j|\\
& \ll U^{-1/2} |Q_+| \frac{\prod_{j\in Q_+}|q_j|}{\min_{j\in
Q_+}|q_j|}\,.
\end{align*}
Hence we have shown the estimate in~\eqref{npl-comp}.
\end{proof}

We now fix $A>2$. {}From here on we view the quantities $t_0>0$ and
$\r\in
(1-\nobreak\tau,1)$ also as absolute quantities, like $\tau$ and $a$
in the sum formula. We apply Lemmas \ref{lem-sf-q}
and~\ref{lem-npl-comp}
to obtain:
\begin{prop}\label{prop-est1q}For $E$, $Q_+$, $Q_-$, $\ph(q,\cdot)$ as
in Lemma~\ref{lem-sf-q}, with $U\geq e^2$:
\begin{align}\label{est1q}
\nct &\left(\ph(q,\cdot)\right) = \frac{2^{|Q|+1}
\sqrt{|D_F|}}{(2\pi)^d} 
\npl_E(\ph_E) \prod_{j\in Q} \nplf_j(q_j)
\displaybreak[0]\\
\nonumber &\qquad\hbox{} + \oh_{F,I,r} \biggl( N_E(\ph_E)
e^{t_0U|Q_+|} \prod_{j\in Q_+} |q_j|^\r \prod_{j\in Q_-} |q_j| ^{-A}
\biggr)
\displaybreak[0]\\
\nonumber &\qquad\hbox{} + \oh \biggl( N_E(\ph_E) U^{-1/2} |Q_+|
\frac{\prod_{j\in Q} |q_j| }{\min _{j\in Q_+} |q_j|} \biggr).
\displaybreak[0]
\end{align}
\end{prop}
This is the basis for the results in the next sections.

The main term in \eqref{est1q} can be estimated by $\oh_F\left(
N_E(\ph_E) \prod_{j\in Q} |q_j|\right)$. This implies
\begin{equation}\label{upb-ctr}
\nct (\ph(q,\cdot)) \ll_{F,I,r,U} N_E(\ph_E) \prod_{j\in Q} |q_j|.
\end{equation}

\section{Upper bound}]\label{sect-ub}
The next step is to derive by integration of \eqref{upb-ctr} an upper
bound for $\nct (f)$ for functions of the form $f=\ph_E\otimes
A \isdef \ph_E\otimes \ch_A$, where $\ch_A$ is the characteristic
function of a set~$A$. To integrate, we fix a non-negative measure
$d_Q$ on $\left( (0,\infty) \cup i[0,\infty)
\right)^Q$ of the form $d_Q q =\bigotimes_{j\in Q} d_j q_j$,
\begin{equation}
\int h(q)\, d_jq = \int_0^\infty h(it)\, dt + \int_0^{\nubs} h(x)\, dx
+ \sum_{\bt>0,\, \bt\equiv\frac{\x_j-1}2(1)} h(\bt).
\end{equation}
We shall use $d_R q = \bigotimes_{j\in R} d_j q_j$ for any set $R$ of
real places.

We define for $b\in\RR$ and for bounded measurable sets $B \subset
\left( \RR\cup i\RR\right)^R$ with $R\subset Q$:
\begin{align}
\nv_b(B) &= \int_B \prod_{j\in R} p(q_j)^b \, d_Rq,\\
\label{p-def} p(q_j)&=\begin{cases}
1&\text{ if } q_j \in (0,\nubs] \cup i[0,1),\\
|q_j|&\text{otherwise}.
\end{cases}
\end{align}
The set $R$ of real places is not visible in the notation $\nv_b$, and
should be clear from the set~$B$. Note that with $b=1$ this
definition agrees with~\eqref{tV1def}.

For our purpose it suffices to estimate $\nct (\ph\otimes A)$ for
bounded sets $A$ of the form
\begin{alignat}2 \label{Adescr}
A&=A_+\times A_0 \times A_-,& \quad A_+ & \subset \left( i[1,\infty)
\right)^{R_+},\\
\nonumber A_0 &= \left( \left( 0,\nubs \right) \cup i[0,1)
\right)^{R_0},&\quad A^-&\subset \prod_{j\in R_-} \left(
\txtfrac{\x_j+1}2+\NN_0\right),
\end{alignat}
for any partition $Q=R_+\sqcup R_0\sqcup R_-$. This choice reflects
that $q$ in \eqref{upb-ctr} has no factors in $(0,\nubs)\cup i[0,1)$.
We have failed to find a test function that allows the sum formula to
see sharp in this region.

The aim is to estimate $\nct (\ph_E\otimes A)$ by
$N_E(\ph_E)\nv_1(A)$. Given perfect knowledge of the spectral set
$\{\nu_{\varpi}\}$ one can choose $A_+$ as the union of tiny boxes
around many $\nu_{\varpi,A_+}=(\nu_{\varpi,j})_{j\in A_+}$ in such a
way that $\ct(\ph_E\otimes A)$ is large while $\nv_1(A)$ stays
arbitrarily small. This shows that we need a further assumption on
the factor $A_+$.

Let $\dist$ be the distance along $i[0,\infty)\cup(0,\nubs)$ given by
\begin{equation}\label{dist-def}
\dist(\nu,q) = \begin{cases} |q-\nu|&\text{ if }q,\nu \in
i[0,\infty)\text{ or }q,\nu\in
(0,\nubs),\\
|q|+|\nu|&\text{ otherwise}.
\end{cases}
\end{equation}
For $\e>0$ and $\nu \in \left( i[0,\infty)\cup(0,\nubs)
\right)^B$, where $B$ is a set of real places:
\begin{equation}
A(\nu,\e) = \left\{ q \in \left( i[0,\infty)\cup(0,\nubs)
\right)^B \;:\; \dist(q_j,\nu_j) \leq \txtfrac\e2\text{ for any }j\in
B\right\}.
\end{equation}
Again, the set $B$ should be clear from the context.

\begin{defn}Let $w,\e>0$, and let $B$ be a set of real places. A
\emph{$(w,\e)$-blunt} subset $H\subset \left(
i[0,\infty)\cup(0,\nubs] \right)^B$ is a $d_Bq$-measurable set such
that
\begin{equation}
\int_{A(\nu_R,\bt) \cap H} d_Bq \geq w \,\vol_B A(\nu,\bt)\quad\text{
for all }\nu \in H\text{ and all }\bt\in(0,\e]\,.
\end{equation}
By $\vol_B$ we mean the volume for $d_B$.
\end{defn}

Note that $(w,\e)$-bluntness implies $(w,\e_1)$-bluntness for any
$\e_1\in (0,\e)$. Boxes with size at least~$\e$ in all coordinate
directions are $(1,\e)$-blunt.

\begin{prop}\label{prop-ub}Let $A=A_+\times A_0 \times A_-$ be as in
\eqref{Adescr}. Suppose that $A_+$ is $(w,\e)$-blunt for some $w>0$
and~$\e\in (0,e^{-1}]$. Then for any $\ph_E$:
\begin{equation}\label{ub}
\ct(\ph_E\otimes A) \ll_{F,I,r} w^{-1} N_E(\ph_E) \nv_1(A).
\end{equation}
\end{prop}
\begin{proof}We apply Proposition~\ref{prop-est1q} with $E$ replaced
by $\hat E = E \cup R_0$, $Q_\pm$ replaced by $R_\pm$, and the test
function $\hat\ph(q,\cdot)$ chosen as follows
\begin{align}
\nonumber &\begin{array}{|c|c|c|c|} \hline j\in& E & R_+\cup R_- & R_0
\\
\hline \hat\ph_j = & \ph_p & \ph_j(q,\cdot) \text{ with }U\geq e^2&
\begin{cases} e^{\nu^2}&\text{ if }|\re\nu|\leq \tau,\\
0&\text{ otherwise};
\end{cases}
\\ \hline
\end{array}\\
\label{phi-p-def} &\ph_p (\nu) = \begin{cases}
(p^2-\nu^2)^{-a/2}&\text{ if }|\re\nu|\leq \tau,\\
(p^2+\nu^2)^{-a/2}&\text{ otherwise},
\end{cases}
\end{align}
with some fixed $p>\tau$ and $q\in \left( i[1,\infty)
\right)^{R_+} \times \prod_{j\in R_-} \left(
\frac{\x_j+1}2+\NN_0\right)$. We put
\begin{equation}
\label{phi-p-E-def} \ph_{p,E}(\nu)=\bigotimes_{j\in E}\ph_p(\nu_j).
\end{equation}
\text{From} \eqref{upb-ctr} we obtain
\[ \nct(\hat \ph(q,\cdot)) \ll_{F,I,r} N_E(\ph_{p,E}) \prod_{j\in
R_+\cup R_-} |q_j|, \] where we have used that \eqref{NE-def} implies
$N_j(\ph_j)=\oh(1)$ for $j\in R_0$. Integration over $q$ gives
\begin{equation}\label{ApAm-int}
\int_{A_+\times A_-} \nct(\hat \ph(q,\cdot)) d_{R_+\cup R_-}q
\ll_{F,I,r} \nv_1(A_+\times A_-).
\end{equation}
We have omitted $N_E(\ph_{p,E})$ since it is $\oh(1)$ for the fixed
choice of~$p$.

Now we note that for a given $\nu \in Y_\x$ we have $\hat\ph(q,\nu)
\geq 0$. With the obvious meaning $(\nu_j)_{j\in B}$ of $\nu_B$ for
sets $B$ of real places, we have
\begin{align*}
\int_{A_+ \times A_-} &\hat\ph(q,\nu) \, d_{R_+\cup R_-} = \int_{A_+
\times\{\nu_{R_-}\}} \hat \ph(q,\nu_\varpi)\, d_{R_+\cup R_-}q
\displaybreak[0]\\
&\geq \int_{A(\nu_{R_+},\e)\cap A_+} \ph_{p,E}(\nu_E) \prod_{j\in R_0}
e^{\nu_j} \prod_{j\in R_-} 1 \prod_{j\in R_+} \ph_j(q_j,\nu_j)\,
d_{R_+}q
\displaybreak[0]\\
& \geq \ph_{p,E}(\nu_E) e^{-|R_0|} \left( \sqrt{\txtfrac
U\pi}e^{-U\e^2} \right)^{|R_+|}\int _{A(\nu_{R_+},\e)\cap A_+} \,
d_{R_+}q
\displaybreak[0]\\
&\geq w \ph_{p,E}(\nu_E) e^{-|R_0|} \biggl( \frac{\e U^{1/2}
e^{-U\e^2}}{\pi^{1/2}}\biggr)^{|R_+|} .
\end{align*}
With the choice $U=\e^{-2}$:
\begin{equation}
\int_{A_+ \times A_-} \hat\ph(q,\nu) \, d_{R_+\cup R_-} \geq w
e^{-|R_0|} \pi^{-|R_+|/2} \ph_{p,E}(\nu_E)\,.
\end{equation}

Since $\hat \ph(q,\cdot)\geq 0$ on $Y_\x$, we can reverse the order of
summation and integration in
\[ \int_{A_+\times A_-} \nct(\hat \ph(q,\cdot))\, d_{R_+\cup R_-}q =
\sum_\varpi \left|c^r(\varpi)\right|^2 \int_{A_+\times A_-} \hat
\ph(q,\nu_\varpi)\, d_{R_+\cup R_-}q. \] Hence
\begin{align*}
\nct(\ph_E &\otimes A) = \sum_{\varpi,\, \nu_{\varpi,Q} \in A}
\left|c^r(\varpi)\right|^2 \ph_{p,E}(\nu_{\varpi,E})
\displaybreak[0]\\
&\leq \frac{e^{|R_0|}}w \pi^{|R_+|/2} \sum_{\varpi,\, \nu_\varpi \in
A} \left|c^r(\varpi)\right|^2 \int_{A_+\times A_-} \hat\ph(q,\nu) \,
d_{R_+\cup R_-}q\displaybreak[0]\\
&\ll_{d} \frac1w \int_{A_+\times A_-} \nct(\hat\ph(q,\cdot))\,
d_{R_+\cup R_-}q
\end{align*}
With \eqref{ApAm-int}:
\begin{equation}
\nct(\ph_E \otimes A) \ll_{F,I,r} \frac1w \nv_1(A),
\end{equation}
where we have used that $\nv_1(A_0)=\oh(1)$.
\end{proof}

\section{Asymptotic formula, first stage} Now we start a more precise
approximation of $\ct(\ph_E\times C)$ where $\ph_E$ is still an
arbitrary test function, and where $C$ is a product $C^+\times C^-$,
with bounded closed sets
$C^+ \subset \left( i[1,\infty) \right)^{Q_+}$ and
$C^-\subset[\frac12,\infty)^{Q_-}$. Only the intersection
$C^-\cap \prod_{j\in Q_-} \left( \frac{\x_j+1}2+\NN_0\right)$ matters
for the present purpose, not $C^-$ itself. For $C^+$ we define, with
$c>0$:
\begin{align}\label{C(c)-def}
C^+(0)&= C^+\,,
\displaybreak[0]\\
\nonumber C^+(c)&= \left\{\nu\in \left( i[0,\infty)\cup(0,\nubs]
\right)^{Q_-}\;:\; A(\nu,c) \cap C^+ \neq \emptyset\right\}\,,
\displaybreak[0]\\
\nonumber C^+(-c)&= \left\{ \nu\in C^+\;:\; A(\nu,c)\subset
C^+\right\}\,,
\displaybreak[0]\\
\nonumber C^+[c]&=C^+(c)\setminus C^+(-c)\,.
\end{align}

\begin{prop}\label{prop-st1-prod}Let $r\in \Ocal'\setminus \{0\}$.
Let $a>2$ and $\tau\in [\frac14,\frac12]$ as in Theorem~\ref{thm-sf},
and decompose the set of real places of~$F$ as
$E \sqcup Q_+ \sqcup Q_-$ with $Q=Q_+\cup Q_-\neq\emptyset$. There
are $t_0>1$, $D>0$, $\r\in (1-\nobreak\tau,1)$ and $A>2$ such that
for any $U>De^2$ and any $\e\in [(\frac{D}U)^{1/2},e^{-1}]$, for all
products $\ph_E=\bigotimes_{j\in E} \ph_j$ of local test functions
and for all bounded $d_Q$-measurable sets $C=C^+\times C^-$,
$C^+\subset \left( i[1,\infty) \right)^{Q_+}$ and $C^-\subset
[\frac12,\infty)^{Q_-}$:
\begin{equation}
\nct (\ph_E \times C) = \Dtfct 
\npl(\ph_E\otimes C) + N_E(\ph_E) \oh_{F,I,r}(E(C,U,\e)) ,
\end{equation}
where if $Q_+=\emptyset$
\begin{equation}\label{E-Q+e} E(C,U,\e) = \nv_{-A}(C^-),
\end{equation}
and if $Q_+\neq\emptyset$
\begin{align}\label{E-Q+ne}
E(C,U,\e)&= e^{t_0U|Q_+|} \nv_\r(C^+) \nv_{-A}(C^-)
+ e^{-U\e^2} \nv_1(C) \\
\nonumber
&\qquad\hbox{} + \nv(C^+[2\e]\times C^-) + U^{-1/2}\nv_1(C).
\end{align}
\end{prop}

At this point we can derive the statement in example~(iv) in the
introduction. We denote by $S_{\!b}(\G,\ch)$ the space of holomorphic
cusp forms on the product ${\mathfrak H}^d$ of $d$ copies of the
upper half plane for the group $\G$ with character $\ch$ and weight
$b\in \NN^d$ satisfying $b_j\geq 1$ and $b_j\equiv\x_j\bmod 2$ for
all~$j$.
\begin{cor}\label{corhol}The space $S_{\!b}(\G_0(I),\ch)$ is non-zero
for all but finitely many weights in the set
$\bigl\{ b\in \NN^d\;:\; b_j\geq 2\,,\;
b_j\equiv\x_j\bmod 2\text{ for all }j\bigr\}$.
\end{cor}
\begin{proof}We apply Proposition~\ref{prop-st1-prod} with
$E=Q_+=\emptyset$, and $C=C^-$ equal to the singleton
$ C_b=\bigl\{ \bigl(\frac{b_j-1}2\bigr)_j\bigr\}$. Then
$\npl(C_b)=\prod_{j=1}^d\frac{b_j-1}2$ and
$\nv_{-A}(C_b)=\prod_{j=1}^d\bigl(\frac{b_j-1}2\bigr)^{-A}$. We
obtain \eqref{holas} with the constant $C$ equal to
$2\sqrt{|D_F|}(2\pi)^{-d}$.

If we take $r$ totally positive, the $\varpi$ entering $\nct(C_b)$
form an orthogonal system of cuspidal representations for which each
factor~$\varpi_j$ is a discrete series representation with lowest
weight~$b_j$. (See (2.29)
in~\cite{BM6+}.) Thus these $\varpi$ correspond to an orthogonal basis
of $S_{\!b}(\G_0(I),\ch)$. So $\nct(C_b)$ can be non-zero only if
$S_{\!b}(\G_0(I),\ch)\neq\{0\}$.
\end{proof}

\subsection{Proof of Proposition \ref{prop-st1-prod}} The proof of the
proposition is rather long and will require some intermediate steps
that we shall give in a series of lemmas.
\subsubsection{Integration}We integrate \eqref{est1q} over~$C$. Taking
into account \eqref{npl-def}, we obtain:
\begin{align}
\int_c \nct &(\ph(q,\cdot))\, d_Q q - \Dtfct 
 \npl_E(\ph_E) \npl_Q(C)
\\
\nonumber &\ll_{F,I,r} N_E(\ph_E) e^{t_0 U |Q_+|} \nv_\r(C^+)
\nv_{-A}(C^-) + N_E(\ph_E) U^{-1/2} V_1(C).
\end{align}
In this term we have left out the denominator $\min_{j\in Q_+}|q_j|$,
since we have already a small factor $U^{-1/2}$.

To prove Proposition \ref{prop-st1-prod} we will estimate the
difference
\begin{equation}\int_c \nct (\ph(q,\cdot))\,
d_Qq - \nct (\ph_e\otimes C).\end{equation}

\subsubsection{Local comparison}\label{sect-loc-cmp}Let $X_j =
X_j^+=\frac{\x_j+1}2+\NN_0$ if $j\in Q_-$, and $X_j =
(0,\nubs]\cup i[0,\infty)$, $X_j^+ = i[1,\infty)$ if $j\in Q_+$. We
consider for $\nu \in X_j$:
\begin{align}\label{IJdef}
I_\al^j(\nu) &= \int_{q\in X_j,\, \dist(q,\nu) \leq \al}
\ph_j(q,\nu)\, d_jq,\\
\nonumber J_\al^j(\nu) &= \int_{q\in X_j,\, \dist(q,\nu) \geq \al}
\ph_j(q,\nu)\, d_jq.
\end{align}
\begin{lem}\label{lem-IJ}If $\al\geq U^{-1/2}$, then for $j\in Q_+$,
$\nu\in X_j$:
\begin{alignat*}3
&\text{ if }\nu\in i[1+\al,\infty):&\quad I_\al^j(\nu)&=1+\oh\left(
e^{-U\al^2}\right),&
\quad J_\al^j(\nu)&= \oh\left( e^{-U\al^2}\right),\\
&\text{if }\nu \in i[0,1+\al):& I_\al^j(\nu)&=\oh(1),& J_\al^j(\nu)
&=\oh\left( e^{-U\al^2}\right).
\end{alignat*}
For $j\in Q_+$, $\nu\in X_j$, $U^{-1/2}\leq \e \leq e^{-1}$, and $\nu
\in i[0,1-\nobreak \e)\cup (0,\nubs]$:
\[ I_\e^j(\nu)=0,\qquad J_\e^j(\nu) = \oh\left( e^{-U\e^2}\right).\]
For $j\in Q_-$, $\nu \in X_j$:
\[ I_\al^j(\nu)=1,\qquad J_\al^j(\nu)=0.\]
\end{lem}
\begin{proof}The results for $j\in Q_-$ are immediate. We consider the
case $j\in Q_+$. The best situation is $\nu \in
i\left[1+\al,\infty\right)$. Then, with \eqref{Gint}:
\begin{align}
I_\al^j(\nu) &= \sqrt{\frac U\pi} \sum_\pm
\int_{|\nu|-\al}^{|\nu|+\al} e^{-U(x\mp|\nu|)^2}\, dx
\displaybreak[0]\\
&\nonumber= \frac 1{\sqrt\pi} \biggl( \int_{-\al\sqrt U}^{\al\sqrt U}
e^{-x^2} + \int_{(2|\nu|-\al)\sqrt U}^{(2|\nu|+\al)\sqrt U}
e^{-x^2}\, dx\biggr)
\displaybreak[0]\\
\nonumber &= 1 + \oh \biggl( \frac{e^{-U\al^2}}{\al\sqrt U} \biggr) +
\oh \biggl( \frac{e^{-U(2|\nu|-\al)^2}}{(2|\nu|-\al)\sqrt U} \biggr)
= 1+ \oh \biggl( \frac{e^{-U\al^2}}{\al\sqrt U} \biggr),
\displaybreak[0]\\
\nonumber J^j_\al(\nu)&= \biggl( \int_1^{|\nu|-\al} +
\int_{|\nu|+\al}^\infty\biggr) \sqrt{\txtfrac U\pi} \sum_\pm
e^{-U(x\mp|\nu|)^2}\, dx
\displaybreak[0]\\
\nonumber &\ll \frac{e^{-U\al^2}}{\al\sqrt U} +
\frac{e^{-U(2|\nu|-\al)^2}}{(2|\nu|-\al)\sqrt U} \ll
\frac{e^{-U\al^2}}{\al\sqrt U}.
\end{align}
We ignore the denominator $\al\sqrt U \geq 1$.

If $\nu \in i[0,1+\nobreak\al]$, part of the integral for
$I^j_\al(\nu)$ is omitted. This leads to the estimate $\oh(1)$. The
quantity $J_\al^j(\nu)$ is at most as large as in the previous case.

We consider finally a small value $\al=\e\in [U^{-1/2},e^{-1}]$. For
$\nu \in i[0,1-\nobreak\e] \cup (0,\nubs]$, the integral
$I_\e^j(\nu)$ vanishes. If $\nu\in i[0,1-\nobreak\e)$ we have already
obtained $J_\e^j(\nu) = \oh\left( e^{-U\e^2}\right)$. For
$\nu \in (0,\nubs]\subset (0,\frac12]$:
\begin{align*}
J^j_\e(\nu) & = \sqrt{ \frac U \pi} \int_1^\infty e^{U(\nu^2-t^2)}\cos
2t\nu\; dt
\displaybreak[0]\\
\nonumber &\ll e^{U\nu^2} \int_{\sqrt U}^\infty e^{-x^2}\, dx \ll
\frac{e^{-U(1-\nu^2)}}{\sqrt U} \leq \frac{e^{-U(1-\nubs^2)}}{1} \leq
e^{-U \e^2}.
\end{align*}
\end{proof}

\subsubsection{Global comparison}
\begin{lem}Let $\nu \in Y_\x$, $\nu_j\in X_j$ for $j\in Q$. Let $\al
\geq U^{-1/2}$ and $\e\in [U^{-1/2},e^{-1}]$. Then
\begin{align}\label{Aest-all}
\int_{q\in A(\nu_{Q_+},\al)\times\{\nu_{Q_-}\}} \ph(q,\nu)\, d_Q q
&\ll \ph_E(\nu_E)\,;\\
\label{Aest-out} \int_{q\in (C^+\setminus
A(\nu_{Q_+},\al))\times\{\nu_{Q_-}\}} \ph(q,\nu)\, d_Q q &\ll
\ph_E(\nu_E) e^{-U\al^2}\,.
\end{align}
If $\nu_j \in i[0,1-\e)\cup (0,\nubs]$ for some $j\in Q_+$ then
\begin{equation}\label{Aest-eps}
\int_{q\in A(\nu_{Q_+},\e)\times\{\nu_{Q_-}\}} \ph(q,\nu)\, d_Q q
=0\,.
\end{equation}
If $A(\nu_{Q_+},\al) \subset C^+$, then
\begin{equation}\label{Aest-C}
\int_{q\in A(\nu_{Q_+},\al)\times\{\nu_{Q_-}\}} \ph(q,\nu)\, d_Q q =
\ph_E(\nu_E) \left(1+ \oh (e^{-U\al^2}) \right)\,.
\end{equation}
\end{lem}
\begin{proof}We have
\[ \int_{q\in A(\nu_{Q_+},\al)\times\{\nu_{Q_-}} \ph(q,\nu)\, d_Q q =
\ph_E(\nu_E) \prod_{j\in Q_+} I_\al^j(\nu_j)\,. \] This implies
directly \eqref{Aest-all}. If $A\nu_{Q_+},\al)\subset C^+$, then
$\nu_j\in i[1+\nobreak\al,\infty)$ for any $j\in Q_+$. Hence
\eqref{Aest-C} follows. Equality \eqref{Aest-eps} follows also from
Lemma~\ref{lem-IJ}.

For \eqref{Aest-out} we use
\begin{align*}
&\int_{q\in (C^+\setminus A(\nu_{Q_+},\al))\times\{\nu_{Q_-}\}}
\ph(q,\nu)\, d_Q q
\displaybreak[0]\\
&\qquad\ll \ph_E(\nu_E) \sum_{m\in Q_+} J^m_\al(\nu_m) \prod_{j\in
Q_+\setminus\{m\}} \left( I_\al^j(\nu_j) + J^j_\al(\nu_j) \right)
\displaybreak[0]\\
&\qquad \ll \ph_E(\nu_E) \sum_{m\in Q_-} \oh(e^{-U\al^2})
\oh(1)^{|Q_+|-1}\,.
\end{align*}
\end{proof}

\subsubsection{Error term} We will use these comparison results to
estimate the following difference:
\begin{align}\label{diffNt}
\nct&(\ph_E\otimes C) - \int_C \nct\biggl( \ph(q,\cdot) \biggr)\, d_Qq
\displaybreak[0]
\\
\nonumber &=\sum_{\varpi,\,\nu_{\varpi_Q} \in C} |{c^r(\varpi)}|^2
  \left( \ph_E\left(\nu_{\varpi,E}\right) - \int_C \ph(q,\nu_\varpi)\,
d_Q q\right) \displaybreak[0]
\\\nonumber
&\qquad\hbox{} - \sum_{\varpi,\, \nu_{\varpi,Q}\not\in C}
|{c^r(\varpi)}|^2 \int_C \ph(q,\nu_\varpi)\, d_Q q,
\end{align}
with $X_j$ as in \S\ref{sect-loc-cmp}.

We write the difference in \eqref{diffNt} as $T_i+T_b+T_o$, given by
the respective conditions $\nu_{\varpi,Q_+} \in C^+(-\e)$,
$\nu_{\varpi,Q_+} \in C^+[\e]$, and $\nu_{\varpi,Q_+} \not\in
C^+(\e)$.

\subsubsection{Inner error term} $C^+(-\e)$ is contained in the subset
\[X_i=\bigcup_{\nu\in C^+(-\e)}A(\nu,\e)\]
of~$C^+$, which is $(1,\e)$-blunt. With \eqref{Aest-C} and
Proposition~\ref{prop-ub}:
\begin{align}
T_i &\ll \sum_{\varpi,\, \nu_{\varpi,Q_+}\in C^+(-\e)} \left|
|c^r(\varpi)|^2 \right| e^{-U\e^2} |\ph_E(\nu_{\varpi,E})|
\displaybreak[0]\\
\nonumber &\ll_{F,I,r} N_E(\ph_E) \nv(X_i\times C^-) e^{-U\e^2} \leq
N_E(\ph_E) e^{-U\e^2}\nv(C)\,.
\end{align}

\subsubsection{Boundary error term}With \eqref{Aest-all}
\[ T_b \ll \sum_{\varpi,\, \nu_{\varpi,Q_+}\in C^+[\e]}
|c^r(\varpi)|^2 |\ph_E(\nu_{\varpi,E})|\,. \] We put $C^+[\e]$ in the
$(1,\e)$-blunt set $\bigcup_{\nu\in C^+[\e]} A(\nu,\e)$ contained in
$C^+[2\e]$. This leads to
\begin{equation}
T_b \ll_{F,I,r} N_E(\ph_E) \nv_1(C^+[2\e]\times C^-)\,.
\end{equation}

\subsubsection{Outer error term}Now we use \eqref{Aest-out}, and use
that the $(1,\e)$-blunt set
\[\bigcup_{\nu\in C^+(\e(n+1))\setminus C^+(\e n)} A(\nu,\e)\]
is contained in $C^+(\e(n+\nobreak2))\setminus C^+(\e(n-\nobreak1))$.
\begin{align}\label{outererror}
T_o &\ll_{F,I,r} N_E(\ph_E) \\
\nonumber &\qquad\hbox{} \cdot \sum_{n=1}^\infty e^{-U\e^2n^2} \left(
\nv_1(C^+(\e(n+2))) - \nv_1(C^+(\e(n-1)))\right)\,.
\end{align}

\subsubsection{Growth on shells} Estimate \eqref{outererror} has the
disadvantage that the bound is given by an infinite sum. Let us
consider $D_n = C^+(\e(n+\nobreak1))
- C^+(\e n)$. The size of the sum
\[\sum_{n=1}^\infty e^{-U\e^2n^2} \nv_1\left( D_{n-1}\cup D_n\cup
D_{n+1}\right)\] depends mainly on the size of $\nv(D_n)$ for small
values of~$n$.

\begin{lem}There is a constant $R = R(|Q_+|) > 1$, not depending on
$\e\in(0,e^{-1}]$, such that $\nv_1(D_n) \leq R^n C^+[\e]$ for any $n
\geq 0$.
\end{lem}
\begin{proof}
The sets $D_n$ are subsets of $\left( i[0,\infty)
\cup(0,\nubs]\right)^{Q_+}$, which we identify with
$[-\nubs,\infty)^{A_+}$: $q_j \in i[0,\infty)$ is replaced by
$q_j\in[0,\infty)$, and $q_j\in (0,\nubs]$ by $-q_j$. Now in each
factor, the distance $\dist$ in \eqref{dist-def} is for each
coordinate given by the absolute value of the difference. The measure
$d_{Q_+}$ corresponds to the Lebesgue measure on $\RR^{Q_+}$.

If $q\in D_n$, then there is $\nu \in C^+$ such that
$\dist(q_j,\nu_j)\leq \e(n+\nobreak1)$ for all $j$ and
$\dist(q_l,\nu_l)>\e n$ for some $l$. For the latter~$l$ we define
$\tilde q_l\in [-\nubs,\infty)$ such that $\dist(\tilde
q_l,\nu_l)\leq \e n$ and $\dist(\tilde q_l,q_l)=\e$. For the other
coordinates we put $\tilde q_j=q_j$.

This implies that each point of $D_{n+1}$ can be moved into $D_n$ by a
translation $T_\vv$ in $\RR^{Q_+}$ over a vector $\vv$ with
coordinates in $\{0,\e,-\e\}$. Hence
\begin{equation} D_{n+1} \subset \bigcup_\vv T_\vv D_n\,.
\end{equation}
There are $3^{|Q_+|}$ translates. For each of these translates
\[ \nv_1(T_\vv D_n) = \int_{D_n } p(x+\vv)\, d_Q x\,.\]
Now we have
\[ p(x+\vv) \leq \sum_{m=0}^{|Q_+|} \binom{|Q_+|}{m} \e^m p(x)\,. \]
Hence we have $ \nv_1(T_\vv D_n) \leq R_1 \nv_1(D_n)$, with $R_1 =
\sum_{m=0}^{|Q_+|}\binom{|Q_+|}{m} e^{-m}$, for any $\e\in
(0,e^{-1}]$. This implies $\nv_1(D_{n+1}) \leq R \nv_1(D_n)$ with
\begin{equation}
R := R(|Q_+|)= \left(3(1+e^{-1})\right)^{|Q_+|}\,.
\end{equation}
Hence $\nv_1(D_n) \leq R^n \nv_1 (D_0) \leq R^n \nv_1(C^+[\e])$.
\end{proof}
The factor $3^{|Q_+|}$ is much too large in most cases, since the
translates $T_\vv D_n$ overlap a lot, and cover more than $D_{n+1}$.

To use this lemma in an estimate of the sum

in \eqref{outererror} we assume that $U\e^2 \geq D$ with $D:=\log R$.
Then $n\mapsto e^{-U\e^2n^2}R^n$ is a decreasing function, and
\begin{align*}
\sum_{n=1}^\infty e^{-U\e^2n^2}& (R^{n+1} + R^n + R^{n-1})
\displaybreak[0]\\
&\leq (R^2+R+1) e^{-U\e^2} + (R+1+R^{-1}) \int_{x=1}^\infty
e^{-U\e^2x^2} R^x\, dx
\displaybreak[0]\\
&\ll_{|Q_+|} e^{-U\e^2} + \frac{e^{-U\e^2+\log C}}{\sqrt{1-\frac{\log
C}{2U\e^2}}} \ll e^{-U\e^2}\,.
\end{align*}
Therefore, under the assumption $U\e^2 \geq D$, where $D=\log(R)$, the
outer error term \eqref{outererror} can be estimated by
\[ \oh_{|Q_+|} \left( e^{-U\e^2} \nv_1(C^+[\e]\times C^-) \right)\,,\]
and hence be absorbed into the term $\oh\left( \nv_1(C^+[2\e]\times
C^-) \right)$. This concludes estimation of the error term, hence the
proof of Proposition~\ref{prop-st1-prod} is now complete.

\subsection{Choice of the parameters $U$ and $\e$}\label{sect-U-eps}
We now arrive at the delicate point where the parameters $U, \e$ will
be linked to the volume quantities, depending on the set $C$. Let us
rewrite the error term $E$ in~\eqref{E-Q+ne}:
\begin{align}
\nonumber E=E(C,U,\e) &=\left( e^{t_0U|Q_+|}m_\r(C) + e^{-U\e^2} +
U^{-1/2} + \bt_\e(C) \right)
\nv_1(C)\,,\\
\label{mr-bt-def} m_\r(C)&= \frac{
\nv_\r(C^+)\nv_{-A}(C^-)}{\nv_1(C)}\,,\qquad \bt_\e(C^+)= \frac{
\nv_1(C^+[2\e])}{\nv_1(C^+)}\,.
\end{align}
We will require that $m_\r(C)$ and $\bt_\e(C)$ get small, to be able
to control the error term in Proposition~\ref{prop-st1-prod}.
Furthermore we will need to choose $U,\e$ suitably. It turns out that
a convenient election will be to let $U$ (resp. $\e$) {\em tend
slowly to $\infty$ (resp. $0$) in such a way that $U\e^2$ still tends
to $\infty$}. Keeping $e^{t_0U|Q_+|} m_\r(C)$ and $U^{-1/2}$ in mind,
we choose
\begin{equation}
U = U(C)=\frac1{t_0|Q_+|}\biggl( |\log m_\r(C)| - \frac12 \log|\log
m_\r(C)| \biggr).
\end{equation}
The condition $U\geq e^2 D$ with $D= \log R(|Q_+|)$, as in
Proposition~\ref{prop-st1-prod}, is satisfied if $m_\r(C)$ is
sufficiently small. With this choice
\[ e^{t_0U|Q_+|}m_\r(C) + U^{-1/2} \ll |\log m_\r(C)|^{-1/2}\,.\]

The contribution $e^{-U\e^2}$ should also be small. We take $\e$ only
 slightly larger than $U^{-1/2}$:
\begin{align}\label{eps-choice}
\e &= \e(C) = \sqrt{\frac{\log|\log m_\r(C)|}{2U}}
\displaybreak[0]\\
\nonumber &= \sqrt{ \frac{t_0|Q_+| \,\log|\log m_\r(C)| }{2 |\log
m_\r(C)| \left(1- \frac12 \frac{\log|\log m_\r(C)| }{ |\log m_\r(C)|
}\right) }}.
\end{align}
The quantity $\e$ tends to zero as $m_\r(C)$ tends to zero, and $\e^2U
=\frac{\log|\log m_\r(C)|}{2}$ tends to $\infty$ as $m_\r(C)$ tends
to $0$. Thus, $\e$ satisfies the conditions in
Proposition~\ref{prop-st1-prod} for sufficiently small values of
$m_\r(C)$.

Let us consider the term $U^{-1/2}\nv_1(C)$ in the error term. With
the choice of $U$ and $\e$ just indicated, this term is slightly
larger than
\[ \frac{\nv_1(C)}{\bigl(\log m_\rho(C) \bigr)^2} \;=\;
\frac{\nv_1(C)}{\bigl( \log
\nv_1(C)-\log(\nv_\rho(C^+)\,\nv_{-A}(C_-)
\bigr)^2}\,.\]
So the size of the error term will in general differ from the size of
the main term by a logarithmic factor. Therefor we switch now from
giving $\oh$-estimates to asymptotic estimates with an $o$-term.

In this way we obtain as the endpoint of the first stage of the
derivation of the asymptotic formula:
\begin{thm}\label{thm-asf1}
Let $r\in \Ocal'\setminus \{0\}$. Divide up the set of real places of
$F$ as $ E \sqcup Q_+ \sqcup Q_-$ with $Q=Q_+\cup Q_-\neq\emptyset$.
Let $\Ccal$ be the collection of bounded $d_Q$-measurable sets
$C=C^+\times C_-$ such that
$C^+\subset \left( i[1,\infty)\right)^{Q_+}$, $C^-\subset
\left[\frac12,\infty\right)^{Q_-}$. Let $\r\in (0,1)$ be as in
Proposition~\ref{prop-st1-prod}, and let $\e(C)$ be as chosen in
\eqref{eps-choice}.

For each product $\ph_E = \bigotimes_{j\in E}\ph_j$ of local test
functions, and for each family $t\mapsto C_t$ in~$\Ccal$ such that as
$t\rightarrow\infty$
\begin{equation}
\begin{aligned}
\label{asf-cond1} \nv_\r(C_t^+)\nv_{-A}(C_t^-) &= o\left(\nv_1
(C_t)
\right)\,,\\
\text{ if }Q_+\neq\emptyset,\text{ then }
 \nv_1(C_t^+[2\e(C_t^+)])\nv_1(C_t^-) &= o\left(\nv_1(C_t)
\right)\,,
\end{aligned}
\end{equation}
the following asymptotic result holds as $t\rightarrow\infty$:
\begin{equation}
\label{asf1} \nct(\ph_E\otimes C_t) = \Dtfct \pl(\ph_E\otimes C_t)
\left(1+o_{F,I,r}(1)\right) \,.
\end{equation}
\end{thm}
For families $t\mapsto C_t$ as in the theorem, the quantities
$\nv_1(C_t)$ and $\npl(C_t)$ have the same size. So we have replace
$o\bigl(\nv_1(C_t)\bigr)$ by $\npl(C_t)$ in~\eqref{asf}.

When formulating the asymptotic formula in $\ld$-space, the quantity
$\nv_b(C_t)$ corresponds to $\v_b(C_t)$, given by the measure
$\v_b = \bigoplus_j \v_{b,j}$,
\begin{align}\label{V-def}
\v_{b,j}(f) &= \frac12\int_{5/4}^\infty f(\ld)(\ld-1/4)^{(b-1)/2}\,
d\ld + \frac12\int_{\lbs}^{5/4}
f(\ld)\,\frac{d\ld}{\sqrt{\left|\ld-\frac14\right|}}\\
\nonumber &\qquad\hbox{} + \sum_{\bt\in \frac{\x_j+1}2+\NN_0} |\bt|^b
f\left( \frac \bt2(1-\frac \bt2)\right)\,.
\end{align}

\subsection{Unions}It is also useful to state an asymptotic formula
for families of disjoint unions $t\mapsto c_t = \bigsqcup_n C(n)_t$
where $C(n)_t = C(n)^+_t\times C(n)^-_t$, with $n$ in a countable
index set. Then we have to replace \eqref{mr-bt-def} by
\begin{align}\label{mr-bt-def-u}
m_\r(C_t) &= \frac1{\nv_1(C_t)} \sum_n
\nv_\r(C(n)^+_t)\nv_{-A}(C(n)^-_t)\,,\\
\nonumber \bt_\e( C_t) &= \frac1{\nv_1(C_t)} \sum_n
\nv_1(C(n)^+_t[2\e])\nv_1(C(n)_t^-)\,.
\end{align}
Proceeding with these choices, we obtain:
\begin{prop}\label{prop-asf1-u}Let $r\in \Ocal'\setminus\{0\}$. Let $E
\sqcup Q$ be a partition of the set of real places of~$F$, with
$Q\neq\emptyset$. Let $t\mapsto C_t$ be a family of bounded
$d_Q$-measurable sets such that for each $t$
\begin{equation}
C_t = \bigsqcup_n C(n)_t\,,
\end{equation}
with each $C(n)_t$ in the collection $\Ccal$ in
Theorem~\ref{thm-asf1}. The decomposition $Q=Q_n^+ \sqcup Q_n^-$ may
depend on~$n$. Under the conditions
\begin{equation}
\begin{aligned}\label{asf-cond1-u}
\sum_n \nv_\r(C(n)^+_t)\nv_{-A}(C(n)^-_t) &= o\left( \pl
(C_t)\right)\,,\\
\text{ if }Q_+\neq\emptyset,\text{ then } \sum_n
\nv_1(C(n)^+_t[2\e(C_t)])\nv_1(C(n)^-_t)
&=o\left( \pl(C_t)\right)\,,
\end{aligned}
\end{equation}
the asymptotic formula \eqref{asf1} holds for each choice of $\ph_E$
as product of local test functions.
\end{prop}

\section{Asymptotic formula, second stage}\label{sect-sst}
We have still the freedom to choose the test function $\ph_E$. In the
second stage we use this freedom to fill in the region
$i[0,1)\cup(0,\nubs]$ for the coordinates of $\nu_\varpi$ in $E$.
More generally, by specializing $\ph_E$ we can make the asymptotic
formula look sharply at the coordinate of $\nu_\varpi$ in~$E$.

We shall choose the test functions $\ph_j$ with $j\in E$ as an
approximation of the characteristic functions of ``intervals'' in
$i[0,\infty) \cup (0,\infty)$. Going over to a description in terms
of the eigenvalue vectors $\ld_\varpi$, we obtain
Theorem~\ref{thm-asf2}, which gives
\[ \ct(B \times \hat C_t^+ \times \hat C_t^-) =
\frac{2\sqrt{|D_F|}}{(2\pi)^d} \pl( B \times \hat C_t^+ \times \hat
C_t^-) \left( 1+ 
o(1)\right)\,,\] where $\hat C_t^+$ and $\hat C_t^-$ are the sets
corresponding to $C_t^+$ and $C_t^-$ under the transformation
$\nu\mapsto \ld$. The set $B$ is a fixed box in~$\RR^E$. This result
is strong, but not adequate for some obvious families.
Proposition~\ref{prop-asf1-u} gives a generalization allowing us to
apply the asymptotic formula to families that are countable disjoint
unions of families of the form $t\mapsto B \times \hat
C_t^+\times \hat C_t^-$.

\subsection{Compactly supported functions at the places in~$E$} For a
family $t\mapsto C_t$ satisfying the conditions in
Theorem~\ref{thm-asf1} or in Proposition~\ref{prop-asf1-u} we rewrite
 \eqref{asf1} as follows:
\begin{equation}
\lim_{t\rightarrow\infty} \Dtfct \; \frac{\nct(\ph_E\otimes
C_t)}{\pl(C_t)} = \npl_E(\ph_E)\,.
\end{equation}

\begin{prop}\label{prop-as-cps}Let $r\in \Ocal'\setminus\{0\}$ and the
decomposition $E\sqcup Q_+\sqcup Q_-$ be as before. For $f_E =
\bigotimes _{j\in E} f_j$ with $f_j :\RR\rightarrow \RR$, define
$\tilde f_E= \bigotimes_{j\in E} \tilde f_j$ by $\tilde f_j(\nu) =
f_j(\frac14-\nobreak \nu^2)$. If $f_j \in C^1_c(\RR)$ for any $j\in
E$, then
\begin{equation}
\lim_{t\rightarrow\infty} \Dtfct \; \frac{\nct(\tilde f_E\otimes
C_t)}{\pl
(C_t)} = 
\pl_E(f_E) .
\end{equation}
\end{prop}
The proof is given in the remainder of this subsection. We can follow
the approach in~\cite{BMP3a} closely.

\subsubsection{Functionals}\label{sect-funct}
First we formulate two lemmas to be used in the proof.

For $f:\RR^E \rightarrow \CC$ put
\begin{equation}\label{Atrrdef}
A^{r}_t(f) = \Dtfct \; \frac{\nct(\tilde f\otimes C_t)}{\npl( C_t)}\,,
\end{equation}
where $\tilde f$ is defined by $\tilde f (\nu)=f\biggl(\bigl(
\frac14-\nu_j^2\bigr)_{j\in Q} \biggr) $. This defines a measure on
$\RR^E$. We want to compare it to the measure $f\mapsto
\npl(\tilde f) = \pl(f)$.
\begin{lem}\label{lem-lim-pos} Let $r\in \Ocal'\setminus\{0\}$, let
$T\mapsto f_T$ be a family of real-valued functions on $\RR^E$, and
let $f$ and $h$ also be real-valued on $\RR^E$, such that
\begin{enumerate}
\item[i)] $f$, $h$ and every $f_T$ is integrable for all $A_t^{r}$ and
for $\pl$.
\item[ii)] $\lim_{t\rightarrow\infty} A_t^{r}(f_T) = \pl(f_T)$ for
all~$T$.
\item[iii)] $\lim_{t\rightarrow\infty} A_t^{r}(h) = \pl(h)$.
\item[iv)] There is a function $T \mapsto a(T)$ such that $a(T)=o(1)$
as $T\rightarrow\infty$, and
\[ |f_T(x)-f(x)| \leq a(T) h(x)
\qquad\text{for all }x\in \RR^E\,.\]
\end{enumerate}
Then $\lim_{t\rightarrow\infty} A_t^{r}(f) = \pl(f)$.
\end{lem}
\begin{proof}The measures $A_t^{r}$ is non-negative, hence
\begin{equation}\label{ieq} A^{r}_t(f_T) - a(T) A_t^{r}(h)
\leq A_t^{r}(f) \leq A^{r}_t(f_T) + a(T) A_t^{r}(h)\,.
\end{equation}
Taking the limit as $t\rightarrow\infty$ of both terms on the side
shows that for all~$T$:
\[ 0 \leq \limsup_{t\rightarrow\infty} A_t^{r}(f)
- \liminf_{t\rightarrow\infty} A_t^{r}(f) \leq 2a(T) A_t^{r}(h)\,. \]
Since $a(T)=o(1)$, the limit $\lim_{t\rightarrow\infty} A_t^{r}(f)$
exists.

With the non-negativity of $\pl$ we derive from~iv) that for all~$T$
\[ \pl(f_T)-a(T) \pl(h)\leq \pl(f) \leq \pl(f_T) + a(T) \pl(h)\,. \]
Hence for all~$T$:
\[ \left| \pl(f) - \lim_{t\rightarrow\infty} A_t^{r}(f) \right| \leq 2
a(T) \pl(h)\,.\] This gives the statement of the lemma.
\end{proof}

\subsubsection{Approximation} Now we start the proof of
Proposition~\ref{prop-as-cps}. For given $f_j\in C_c^1(\RR)$ we take
\begin{equation}
\ph_j(\nu) = \sqrt{\frac T \pi} \int_{-\infty}^\infty
e^{-T\left(\ld-\frac14+\nu^2\right)^2} f_j(\ld)\, d\ld,
\end{equation}
with $T$ a large positive parameter. Again we use a Gaussian kernel
functions, but now in the $\ld$-space. This defines $\ph_j(\nu)$ as
an even holomorphic function of $\nu \in \CC$, with exponential decay
on the strip $|\re\nu|\leq \tau$, and for $\nu\in \RR$. So $\ph_j$ is
a test function in the sense of~\S\ref{sect-tf}. For
$\nu\in \RR\cup i\RR$, it is given by
\begin{equation}
\ph_j(\nu) = \sqrt{\frac T \pi} \int_{-\infty}^\infty e^{-T\ld^2} f_j
\left( \ld+\txtfrac14-\nu^2\right)\, d\ld.
\end{equation}

We view $\ph_j$ as an approximation of $\tilde f_j$. Similarly,
$\ph_E$ is an approximation of $\tilde f_E $.

\subsubsection{Local estimates} Since $f_j$ is real-valued, we have $
|\ph_j(\nu)| \leq \|f_j\|_\infty$ for any $\nu \in \RR\cup i\RR$.
Take $N$ large, such that $\supp f_j \subset [-N,N]$ for any $j\in
E$. If $\nu \in \RR\cap i\RR$ with $|\nu|\geq \sqrt{2N+1}$, and
$-N\leq \ld \leq N$, then
\[ \ld-\frac14+\nu^2 \begin{cases}
\leq N- \frac14-|\nu|^2 \leq -\frac12|\nu|^2
- \frac34&\text{ if }\nu\in i\RR,\\
\geq -N-\frac14 +|\nu|^2 \geq \frac12|\nu|^2+ \frac14&\text{ if }\nu
\in \RR.
\end{cases}
\]
Hence $e^{-T\left( \ld-\frac14+\nu^2\right)^2} \leq e^{-\frac14 T
\nu^4}$ for such values of $\nu$ and~$\ld$. Together, these facts
give for $\nu \in \RR\cup i\RR$:
\begin{equation}
|\ph_j(\nu)| \leq \begin{cases}
\|f_j\|_\infty&\text{ if } |\nu|<\sqrt{2N+1},\\
2N \sqrt{\frac T \pi} e^{-\frac14 T \nu^4} \|f_j\|_\infty&\text{ if
}|\nu|\geq \sqrt{2N+1}.
\end{cases}
\end{equation}

We recall the positive test function $\ph_{p,E}= \bigotimes_{j\in E}
\ph_p$ in \eqref{phi-p-E-def}, which gives for $p>\tau$:
\[ \ph_p(\nu_j) = \begin{cases}
(p^2-\nu_j^2)^{-a/2}&\text{ if }|\re\nu|\leq \tau,\\
(p^2+\nu_j^2)^{-a/2}&\text{ otherwise}.
\end{cases}
\]
We take $T\geq T_0=4$. For $\nu\in \RR\cup i\RR$, $|\nu|\geq
\sqrt{2N+1}$, we have $T e^{-\frac14T |\nu|^4} \leq T_0 e^{-|\nu|^4}
   \ll |\nu|^{-a} $, and hence
   \begin{equation}
|\ph_j(\nu)| \ll T^{-1/2} N \|f_j\|_\infty |\nu|^{-a} \ll_N T^{-1/2}
\ph_p(\nu).
\end{equation}

For $\nu\in \RR\cup i\RR$, $|\nu|\leq \sqrt{2N+1}$:
\begin{align}
\ph_j(\nu) &- \tilde f_j(\nu) = \frac1{\sqrt\pi} \int_{-\infty}^\infty
e^{-y^2} \left( f_j\left(\txtfrac y {\sqrt
T}+\txtfrac14-\nu^2\right)-f_j\left(\txtfrac14-\nu^2\right)
\right)\, dy
\displaybreak[0]\\
\nonumber &\ll \|f'_j\|_\infty \int_0^{T^{1/4}} e^{-y^2} \frac y{\sqrt
T}\, dy + \|f_j\|_\infty \int_{T^{1/4}}^\infty e^{-y^2}\, dy
\displaybreak[0]\\
\nonumber &\ll T^{-1/2} \|f'_j\|_\infty + \frac{ \|f_j\|_\infty
e^{-T}}{T^{1/4} } \ll T^{-1/2} \left( \|f_j\|_\infty +
\|f_j'\|_\infty\right).
\end{align}

Under the assumption $|\nu|\leq \sqrt{2N+1}$, $\nu\in \RR\cup i\RR$,
we have $\ph_p(\nu) \gg N^{-a/2}$. Hence
\begin{equation}\label{estbyphpj}
\ph_j(\nu)-\tilde f_j(\nu) \ll_{f_j} T^{-1/2} N^{a/2} \ph_p(\nu).
\end{equation}

\subsubsection{Global approximation} We apply the
Lemma~\ref{lem-lim-pos} with
\begin{alignat*}2 f_T(\ld)&= \ph_E(\nu)\,,&
h(\ld)&=\ph_{p,E}(\nu)\,,\\
f(\ld)&=f_E(\ld) =\bigotimes_{j\in E} \tilde f_j(\nu_j)\,,
\end{alignat*}
with $\ld_j=\frac14-\nu_j^2$, $\nu_j = \pm\sqrt{\frac14-\ld_f}$.
Condition~i) is satisfied by continuity. We have ii) and iii) from
the assumption that $t\mapsto C_t$ is a family for which
Theorem~\ref{asf1} holds. To check Condition~iv) we note that if $\nu
\in \left( \RR\cup i\RR \right)^{E}$, such that $|\nu_j|\leq
\sqrt{2N+1}$ for all $j$, then:
\begin{equation}
\left|\ph_E(\nu) - \tilde f_E(\nu)\right| \ll_f T^{-1/2}
\ph_{p,E}(\nu).
\end{equation}
If there is at least one $j\in E$ with $|\nu_j|\geq \sqrt{2N+1}$, then
by (\ref{estbyphpj})
\begin{equation}\label{estbyphpE}
\left|\ph_E(\nu) - \tilde f_E(\nu)\right| = \left|\ph_E(\nu)\right|
\ll_f T^{-1/2} \ph_{p,E}(\nu).
\end{equation}

Application of the lemmas in \S\ref{sect-funct} completes the proof of
Proposition~\ref{prop-as-cps}.

\subsubsection{Remark}We refrain from extending the asymptotic formula
to compactly supported functions on $\RR^E$ that have no product
structure.

\subsection{Boxes}Proposition~\ref{prop-as-cps} works with compactly
supported continuous functions with product structure. The last step
in stage two is the extension to boxes in $\RR^E$. We now formulate
the asymptotic formula in terms of the coordinate $\ld$, and use the
notation
\begin{equation}
\hat C = \left\{ (\txtfrac14-\nu_j^2)_{j\in Q}\;:\; \nu \in C
\right\}\,,
\end{equation}
for $C \subset \left( i\RR\cup\RR\right)^{Q}$.
\begin{thm}\label{thm-asf2}
Let $r\in \Ocal'\setminus\{0\}$. Let $E \sqcup Q_+ \sqcup Q_-$ a
decomposition of the real places of~$F$ with $Q=Q_+\cup
Q_-\neq\emptyset$. Let $t\mapsto C_t$ be a family of bounded
$d_Q$-measurable sets in the collection $\Ccal$ in
Theorem~\ref{thm-asf1} or as considered in
Proposition~\ref{prop-asf1-u}. In particular we suppose that the
conditions in \eqref{asf-cond1} or in \eqref{asf-cond1-u} hold. Let
$B_E = \prod_{j\in E} [A_j,B_j]$ be such that
\begin{equation}\label{ABc}
A_j,B_j \not\in \left\{ \txtfrac b2( 1-\txtfrac b2)\;:\; b>1,\,
b\equiv \x_j\bmod2\right\}\,.\end{equation}
 Then, as $t\rightarrow\infty$:
\begin{equation}\label{asf-final}
\ct(B_E \times \hat C_t) = \Dtfct \; \pl(B_E \times \hat C_t) +
o\bigl( \pl_Q(\hat C_t) \bigr) \,.
\end{equation}
\end{thm}
\begin{proof}Let $\ch_E$ be the characteristic function of $B_E$. It
has the form $\ch_E = \bigotimes_{j\in E} \ch_j$. The local
characteristic functions $\ch_j$ are integrable for $\pl_{\x_j}$ and
for all $A_t^{r}$ as defined in \eqref{Atrrdef}.

The conditions in the theorem on the endpoints $A_j$ and $B_j$ make it
possible to find for a given $T\geq 1$ elements $u_{T,j},
U_{T,j}\in C_c^1(\RR)$ such that $0\leq u_{T,j} \leq \ch_j \leq
U_{T,j}$ on $\RR$, and such that
\begin{equation}
\pl_{\x_j}(U_{T,j}-u_{T,j})) \leq \frac1T\,.
\end{equation}

Put $u_T = \bigotimes_{j\in E} u_{T,j}$, $U_T = \bigotimes_{j\in E}
U_{T,j}$, and $h=\bigotimes_{J\in E} h_j$. The asymptotic formula
holds for $h$ and for all $u_T$
(Proposition~\ref{prop-as-cps}). We have
\begin{align}
\nonumber 0 &\leq U_T(\ld) - u_T(\ld) \leq \sum_{m\in E}
(U_{T,m}(\ld_m)-u_T(\ld_m) )
\prod_{j\in E \setminus\{m\}} h_j(\ld_j)\,,\displaybreak[0]\\
\label{Plchu} 0&\leq \pl_E(U_T-u_T) \leq \sum_{m\in E} T^{-1}
\prod_{j\in E \setminus\{m\}} \pl_{\x_j}(h_j) = \oh(T^{-1})\,.
\end{align}

Since the $A_t^{r}$ are non-negative measures, we have:
\begin{align*}
A^{r}_t( u_T) &\leq A_t^{r}(B_E) \leq A_t^{r}(U_T)\,,
\displaybreak[0]\\
\pl_E(u_T) &\leq \liminf_{t\rightarrow\infty} A_t^{r}(B_E) \leq
\limsup_{t\rightarrow\infty} A_t^{r}(B_E) \leq \pl_E(U_T)\,,
\displaybreak[0]\\
\limsup_{t\rightarrow\infty} &A_t^{r}(B_E) -
\liminf_{t\rightarrow\infty} A_t^{r}(B_E) = \oh(T^{-1})\,.
\end{align*}
So $\lim_{t\rightarrow \infty} A_t^{r}(B_E)$ exists, and satisfies
\[ \pl_E(u_T) \leq \lim_{t\rightarrow \infty} A_t^{r}(B_E) \leq
\pl_E(U_T)\,.\] Again applying \eqref{Plchu} we conclude that this
limit is equal to $\pl_E(B_E)$. Hence $\ct(B_E\times \hat C_t) =
\Dtfct\, \pl(B_E\times\hat C_t) + o\bigl( \pl_Q(\hat C_t)
\bigr)$, which is~\eqref{asf-final}.
\end{proof}

\section{Special families}\label{sect-special}Theorem~\ref{thm-asf2}
describes a large class of families of sets for which the asymptotic
formula \eqref{asf-final} holds. It has the limitation that the
factor $C_t$ is a subset of $\left(i[1,\infty) \cup
[1/2,\infty) \right)^Q$, while the region $i[0,1)\cup(0,\nubs]$ is
treated only in the coordinates in~$E$. To avoid technical
complications we have chosen not to try to derive an asymptotic
formula for a larger class of families of sets, but to apply
Theorem~\ref{thm-asf2} in a number of special cases. This will
suffice to give many applications.

\subsection{Boxes}Directly from Theorem~\ref{thm-asf2} we get families
of boxes of the type $\tilde B_E \times C_t$ with
\begin{equation}\label{C-box}
C_t^+ = \prod_{j\in Q_+} i[a_j(t),b_j)t)]\,,\qquad C_t^- = \prod_{j\in
Q_-} [a_j(t),b_j(t)]\,,
\end{equation}
where for all $t$
\begin{equation}\label{ab-ineq}
1\leq a_j(t) \leq b_j(t) \quad\text{ if }j\in Q_+\,,\qquad \frac12
\leq a_j(t) \leq b_j(t) \quad\text{ if }j\in Q_-\,.
\end{equation}
A computation of the quantities in \eqref{mr-bt-def} shows that
\begin{align}\label{mr-eps-box}
m_\r(C_t) &\ll \prod_{j\in Q_+} b_j(t)^{\r-1} \prod_{j\in Q_-}
\frac1{\nv_1(C_t^-)}\,,\\
\nonumber \bt_\e(C_t^+) &\ll \e \sum_{m\in Q_+}
\frac{b_m(t)}{(a_m(t)+b_m(t) )(b_m(t)-a_m(t)+\e)}\,.
\end{align}
In the uninteresting case when $C_t^-$ does not intersect $\prod_{j\in
Q_-} \left( \frac{\x_1+1}2+\NN_0\right)$, we have $\nct(B_E\times
C_t)= \npl(B_E\times C_t)=0$. So we assume that $Q_-=\emptyset $ or
$\nv_1(C_t^-)>0$. Then $\nv_1(C_t^-)\geq 1$ if $Q_- \neq\emptyset$.

The conclusion is that $m_\r(C_t)\downarrow0$ as soon as for at least
one $j\in Q$ we have $b_j(t)\rightarrow\infty$.

With $\e=\e(C_t) $ as in \eqref{eps-choice} it suffices to require in
\eqref{mr-eps-box} that for any $m\in Q_+$:
\[\left( b_m(t)-a_m(t)+\e\right) =o(1)\,.\]
This can be achieved by requiring that $b_j(t)-a_j(t) \geq \g|\log
m_\r(C_t)|$ for any $j\in Q_+$ and all $t$ large, for any $\al\in
(0,\frac12)$ and any~$\g>0$. Thus we have:
\begin{prop}\label{box-restr}The asymptotic formula holds for a family
of boxes $t\mapsto \tilde \Om = \tilde B_E \times C_t$ with $B_E$ any
box in $ \RR^E$ satisfying \eqref{ABc} and $C_t$ as in \eqref{C-box}
and \eqref{ab-ineq} under the conditions
\begin{enumerate}
\item[a)] $b_j(t) \rightarrow\infty$ for some $j\in Q$.
\item[b)] $b_j(t)-a_j(t) \geq \s(t)$ for any $j\in Q_+$ and all~$t$,
with
\[ \s(t) = \g \biggl( (1-\r)\sum_{j\in Q_+} \log b_j(t)+ \log
\nv_1(C_t^-)\biggr)^{-\al}\qquad(\g>0\,,\; 0<\al<\frac12)\,. \]
\item[c)] $[a_j(t),b_j(t)] \cap \left( \frac{\x_j+1}2+\NN_0\right)
\neq\emptyset$ for all $t$ and for any $j\in Q_-$.
\end{enumerate}
\end{prop}

This proposition implies Theorem~\ref{thm-B-H-p}, and its consequences
Propositions \ref{prop-hypercube} and~\ref{prop-holo}.
\medskip

For boxes in the $\ld$-parameter we have the following result:
\begin{prop}The asymptotic formula holds for families of boxes
$t\mapsto \Om_t$, where
\[ \Om_t = \prod_j [A_j(t),B_j(t)] \]
satisfies the following conditions:
\begin{enumerate}
\item[ a)] If for a fixed $j$ $A_j$ and $B_j$ are constant then
$A_j=A_j(t)$ and $B_j=B_j(t)$ satisfy condition~\eqref{ABc}.
\item[b)] If $A_j(t)$ is not constant, then $A_j(t)\leq 0$ for
all~$t$, or $A_j(t) \geq \frac54$ for all~$t$. Similarly for
$B_j(t)$.
\item[c)] There is a constant $\s>0$ such that if $B_j(t)\geq \frac54$
then
\[ B_j(t) - \max(A_j(t),\txtfrac54) \geq \s\biggl( \sqrt{|B_j(t)} +
\sqrt{\max(A_j(t),\txtfrac54)} \biggr)\,.\]
\item[d)] If $A_j(t)\leq 0$, then the interval $[A_j(t),B_j(t)]$
intersects for all $t$ the set of $\frac b2(1-\nobreak\frac b2)$,
$b>1$, $b\equiv \x_j\bmod 2$ non-trivially for all~$t$. (This
intersection may depend on~$t$.)
\item[e)] $\lim_{t\rightarrow\infty} A_j(t)=-\infty$ or
$\lim_{t\rightarrow\infty} B_j(t)=\infty$ for at least one~$j$.
\end{enumerate}
\end{prop}
This is not the most general statement for boxes. We have decided not
to complicate the proposition by considering non-constant endpoints
that have values in $(0,\frac 54)$.
\begin{proof}Let $E_0$ be the set of places for which $A_j$ and $B_j$
are constant. We consider partitions $Q_+ \sqcup Q_0 \sqcup Q_-$ of
the remaining infinite places of~$F$. For each of these partitions
$P$ we form
\[ \Om_t^P = \Om_t \cap \left(\RR^{E_0} \cup \left( i[1,\infty)
\right)^{Q_+} \cup \left( i[0,1)\cup(0,\nubs]\right)^{Q_0} \cup
[\txtfrac12,\infty)^{Q_-}\right)\,.\]

Suppose $Q_+ \neq \emptyset$. For $j\in Q_+$ we write $A^\cdot_j(t) =
\max(A_j(t),\frac54)$. In the $\nu$-description, the factor
$\Om_{t,j}^P$ is of the form $i[a_j(t),b_j(t)]$ with
$a_j(t) = \sqrt{A^\cdot_j(t)-\frac14}$ and $b_j(t) =
\sqrt{B_j(t)-\frac14}$. Condition~c) implies that condition~b) in
Proposition~\ref{box-restr} is satisfied.

If $Q_-\neq \emptyset $ for $P$, then condition~d) implies
condition~c) in Proposition~\ref{box-restr}.

For the partition~$P$ we take $E=E_0\cup Q_0$, and try to apply
Proposition~\ref{box-restr} to $t\mapsto \Om_t^P$. This gives the
asymptotic formula for $\Om_t^P$ provided either there is $j\in Q_+$
for which $B_j(t)\rightarrow\infty$, or there is $j\in Q_-$ for which
$A_j(t)\rightarrow-\infty$. Otherwise, the set $\Om_t^P$ is bounded.

Condition~e) implies that the asymptotic formula holds for at least
some partition~$P$. Thus $\nv_1(\Om_t^P)\rightarrow\infty$ for such
$P$. Adding the corresponding finitely many terms we get the
asymptotic formula for the union of the $\Om_t^P$. For the remaining
partitions ~$P$, the set $\Om_t^P$ stays bounded. Adding the
corresponding terms to the asymptotic formula does no harm. This
gives the asymptotic formula for $t\mapsto \Om_t$.
\end{proof}

Now an approximation of $\pl\left( [-X,X]^d\right)$ gives
Proposition~\ref{prop-Weyl1}.

\subsection{Simplices} The results in \cite{BMP3a} are for sets of the
form
\begin{equation}
\hat C_t = \bigl\{ \ld \in [0,\infty)^{Q_+} \times
(-\infty,0)^{Q_-}\;:\; \sum_{j\in Q} |\ld_j|\leq t\bigr\}\,.
\end{equation}

By showing that the asymptotic formula holds for sets of this form, we
extend the results in \cite{BMP3a} to general character~$\ch$ and
general compatible central character given by~$\x$.
\begin{prop}\label{prop-simp}Let $E\sqcup Q_+ \sqcup Q_-$ be a
partition of the infinite places of $F$ with $Q=Q_+\cup
Q_-\neq\emptyset$. Let $B_E$ be any box in $\RR^E$ satisfying
\eqref{ABc}. The asymptotic formula holds for $t\mapsto B_E \times
\hat C_t$, and
\begin{equation}
\pl(\hat C_t) \sim \frac1{ |Q|!} t^{|Q|}
\qquad(t\rightarrow\infty)\,.
\end{equation}
\end{prop}
\begin{proof}Let us first consider
\[ W_n(Y) = \bigl\{ \ld\in \bigl[\txtfrac54,\infty\bigr)^n\;:\; \sum_j
\ld_j \leq Y \bigr\}\,. \] We have $\v_1(W_1(Y)) = \frac12(Y-\nobreak
\txtfrac 54)$ for $Y \geq \frac 54$. {}From
\[ \v_1(W_n(Y)) = \frac12 \int_{5/4}^Y
\v_1\bigl((W_{n-1}(Y-\nobreak\ld) \bigr) \, d\ld\,,\] we obtain by
induction that:
\begin{equation}
\v_1(W_n(Y)) = \frac1{2^n\; n!} \bigl( Y - \txtfrac 54 n\bigr)_+^n\,.
\end{equation}
We use $(x)_+ = 0$ if $x<0$, and $(x)_+=x$ if $x\geq 0$.

For $\r\in(0,1)$ we find by the inclusion $W_n(Y) \subset [5/4,Y]^n$
that
\begin{equation}
\v_\r\bigl(W_n(Y)\bigr) = \oh_n\bigl( Y^{n(\r+1)/2}\bigr)\,.
\end{equation}

Furthermore, for $\e$ small in comparison with $Y$, the part of
$W_n(Y)[2\e]$ on which $ \ld_j>\frac 54$ for all $j\in Q_+$ is
contained in $W_n(Y_+) - W_n(Y_-)$ with $Y_+ =Y+4\e n Y^{1/2}+4\e^2n$
and $Y_-=Y-4\e nY^{1/2}$. The other parts of $W_n(Y)[2\e]$ are
contained in boxes of the form $\left[\frac 54-2\e,
Y+2\e\right]^{n-1} \times \left[ \frac 54-2\e,\frac 54+\e\right]$.
Hence
\begin{align}
\nv_1\bigl( &\tilde W_n(Y)[2\e]\bigr) \leq
\v_1(W_n(Y_+))-\v_1(W_n(Y_-))
 + n \oh(Y^{n-1} \e)
 \\
\nonumber &\leq \frac1{2^n\;n!} \left( (Y+4\e n Y^{1/2} +4\e^2n -
\txtfrac 54 n)^n -
(Y-4\e nY^{1/2}- \txtfrac 54 n)^n\right) \\
\nonumber &\qquad\hbox{}
+ \oh_n(\e Y^{n-1}) \\
\nonumber &\ll_n \e Y^{n-\frac12}\,.
\end{align}
If $Y-\frac54n$ is small, we get at least $\oh(\e)$, which is $\oh(\e
Y^{n-1/2})$ as well.

Now we apply Proposition~\ref{prop-asf1-u} to the following subset of
$\hat C_t$:
\begin{equation}
W_t= \bigsqcup_\pp W_{|Q_+|}\biggl( t - \sum_{j\in
Q_-}\bigl(\pp_j^2-\txtfrac14\bigr)
 \biggr) \times \{\pp\}\,,
\end{equation}
where $\pp$ runs over $\prod_{j\in Q_-} \left( \frac{3-\x_j}2+\NN_0
\right)$ for which $ \sum_{j\in Q_-} \bigl(\pp_j^2-\frac14\bigr) \leq
t$. For each given $t$, this is a finite union. But as a family
depending on~$t$ it is an infinite union. \nonumber We obtain
\begin{align}
\v_1(W_t) = \sum_\pp &\frac1{2^{|Q_+|}\;|Q_+|!} \biggl( t -
\txtfrac54|Q_+|+ \txtfrac14 |Q_-| - \sum_{j\in Q_-}
\pp_j^2\biggr)^{|Q_+|}_+ \prod_{j\in Q_-}\pp_j\,,\\
\nonumber m_\r(W_t)\v_1(W_t) &\ll_{Q_+} \sum_\pp \biggl( t +
\txtfrac14 |Q_-| - \sum_{j\in Q_-} \pp_j^2\biggr)^{|Q_+|(\r+1)/2}
\prod_{j\in Q_-} \pp_j^{-A}\,,\\
\nonumber \bt_\e(W_t) \v_1(W_t)& \ll_{Q_+} \e \sum_\pp \biggl( t +
\txtfrac14 |Q_-| - \sum_{j\in Q_-} \pp_j^2\biggr)^{|Q_+|-\frac12}\,.
\end{align}

We compare the sum for $\v_1(W_t)$ with the integral
\begin{align*}
&\int_{\xx \in [1,\infty)^{Q_-},\, \sum_{j\in Q_-}\xx_j^2\leq
t+\frac14|Q_-|} \left( t+\txtfrac14|Q_-|-\txtfrac 54 |Q_+| -
\sum_{j\in Q_-}\xx_j^2\right)^{|Q_+|} \\
&\qquad\qquad\qquad\hbox{} \cdot\prod _{j\in Q_-} \xx_j \, d\xx
\displaybreak[0]
\\
&\quad= \int_{\yy \in [5/4,\infty)^{Q_-},\, \sum_{j\in Q_-} \yy_j\leq
t} \left( t - \txtfrac54 |Q_+|-\sum_{j\in Q_-} \yy_j
\right)^{|Q_+|}\\
&\qquad\qquad\qquad\hbox{} \cdot 2^{-|Q_-|}\, d\yy
\displaybreak[0]
\\
&\quad= \frac1{2^{|Q|} \,|Q|!} \left( t-\txtfrac54 |Q|\right)^{|Q|}
\sim \frac{t^{|Q|}}{2^{|Q|} \, |Q|!}\,.
\end{align*}
The transition from sum to integral gives a contribution
$\oh(t^{|Q|-1})$.

The other sums can be treated similarly. This leads to the estimates
$m_\r(W_t) \ll t^{|Q_+|(\r-1)/2}$ and $\bt_\e(W_t) \ll t^{-1/2}$.
This implies that the asymptotic formula holds for
$t\mapsto B_E \times W_t$.

The definition of $\npl$ shows that $\pl(W_t) \sim
\frac{2^{|Q|}}{2^{|Q|} \; |Q|!} \v_1(W_t)$.

The remaining parts in $\hat C_t \setminus W_t$ are contained in
unions of sets of the form $Y=[0,\frac54] \times [-t,t]^{|Q|-1}$ for
which $\v_1(Y) \ll t^{|Q|-1}$ and $\pl(Y) \ll t^{|Q|-1}$. So adding
these parts do not influence the asymptotic formula or the asymptotic
behavior of $\pl(\hat C_t)$.
\end{proof}
Proposition~\ref{prop-Weyl2} is a corollary of this result.

\subsection{Sectors}Under the assumption $|Q_+|=2$ we consider
\begin{equation}
S_{\!p,q,\al,t} = \left\{ (\ld_1,\ld_2)\in \left(
\txtfrac54,\infty\right)^2\;:\; t\leq \ld_1\leq t+t^\al\,, p\ld_1\leq
\ld_2 \leq q\ld_2\right\}\,,
\end{equation}
with $0<p<q$, $\al\leq 1$, and $t\rightarrow\infty$.
\begin{lem}For $c\leq 1$:
\begin{equation}
\label{Vc-sect}\v_c(S_{\!p,q,\al,t}) \sim
\frac1{2(c+1)}\left(q^{\frac{c+1}2}-
p^{\frac{1+c}2}\right)t^{c+\al}\,. \end{equation} Furthermore, for
$\e>0$ small in comparison to $t$ as $t\rightarrow\infty$:
\begin{align}
\label{mr-sect} \v_\r(S_{\!p,q,\al,t}) &\ll_{p,q} t^{\r-1}
\v_1(S_{\!p,q,\al,t})\,,\\
\label{bd-sect} \v_1(\tilde S_{\!p,q,\al,t}[2\e]) &\ll_{p,q} \e
t^{\max(\frac 32 ,\al+\frac12)} \v_1(S_{\!p,q,\al,t})\,.\end{align}
\end{lem}
\begin{proof}\eqref{Vc-sect} is obtained by direct computation based
on \eqref{V-def}. It immediately implies \eqref{mr-sect}.

The description of $S_{\!p,q,\al,t}$ in $\nu$-space is as follows:
\[ \tilde S_{\!p,q,\al,t} = \left\{ i(t_1,t_2)\;:\; a\leq t_1\leq
b\,,\; \sqrt{pt_1^2+\txtfrac{p-1}4} \leq t_2
\leq\sqrt{qt_1^2+\txtfrac{q-1}4}\right\}\,, \] with
$a=\sqrt{t-\frac14}$, $b=\sqrt{t+t^\al-\frac14}$. The left hand side
of $\tilde S_{\!p,q,\al,t}[2\e]$ is contained in
\[ [ a-2\e,a+2\e] \times \bigl[
\sqrt{pt-1/4}+2\e,\sqrt{qt-1/4}-2\e\bigr]\,.\] Its contribution to
$\nv_1(\tilde S_{\!p,q,\al,t}[2\e])$ is $\oh_{q-p}\left( \e a \cdot
t \right) = \oh(t^{3/2})$. The contribution of the right hand side of
$\tilde S_{p,q,\al,t}[2\e]$ has the same order.

On the lower side remains a piece for which $\nv_1$ can be estimated
by
\[ \int_{t_1=a}^b t_1 \int_{t_2 =
\sqrt{pt_1^2+\frac{p-1}4}-2\e(1+p)}^{\sqrt{pt_1^2+\frac{p-1}4}+2\e(1+p)}
t_2\, dt_2\, dt_2\,. \] The inner integral can be estimated by $\oh_p
\left(\e\sqrt{pt_1^2+\frac{p-1}4}\right)$. This gives for the total
integral
\begin{align*}
&\ll_p \e \left( \bigl(pb^2+\txtfrac{p-1}4\bigr)^{3/2} -
\bigl(pa^2+\txtfrac{p-1}4\bigr)^{3/2} \right)\\
&\ll_p \e (pb^2-pa^2) \bigl(pb^2+\txtfrac{p-1}4\bigr)^{1/2} \ll_p \e
t^{\al+\frac12}\,.
\end{align*}
Similarly, the upper part contributes $\oh_q(\e t^{\al+1/2})$.
\end{proof}
This lemma shows that if $\al\geq \frac12$, the asymptotic formula
holds for the family $t\mapsto B_E \times S_{\!p,q,\al,t}$ for any
choice of the box $B_E$ as before, where $E$ contains all infinite
places except the two places we put in~$Q_+$. In particular, we
obtain Proposition~\ref{prop-sector}.

\subsection{Spheres} We consider the sphere $S_{\!Q_+}(\mm,r) \subset
\left( i[1,\infty)\right)^{Q_-}$ with radius~$r$ and center $\mm$:
The set of $\nu$ with $\sum_{j\in Q_+} \left( |\nu_j|
-|\mm_j|\right)^2 \leq r^2$. We suppose that $|\mm_j|-r\geq 1$ for any
$j\in Q_+$, and that $|\mm_j|\rightarrow\infty$ for at least one
$j\in \QQ_+$. A computation by induction on $|Q_+|$ leads to
\begin{equation} \nv_1\left( S_{\!Q_+}(\mm,r)\right) = 2 v_{|Q_+|}
r^{|Q_+|} \prod_{j\in Q_+} |\mm_j|\,,
\end{equation}
where $v_n$ is the volume of the unit sphere in $\RR^n$.

Furthermore
\begin{align}
\nv_\r\left( S_{\!Q_+}(\mm,r) \right) &\ll r^{|Q_+|} \prod_{j\in
Q_+}|\mm_j|^\r\,,\\
\nonumber \nv_1\left( S_{\!Q_+}(\mm,r)[2\e] \right) &\ll \e r^{n-1}
\prod_{j\in Q_+}|\mm_j|\,.
\end{align}
These estimates follow from the inclusion
\begin{align*}
S_{\!Q_+}(\mm,r) & \subset \prod_{j\in Q_+} i[|\mm_j|-\nobreak
r,|\mm_j+\nobreak r]\,,\\
S_{\!Q_+}(\mm,r) [2\e] &\subset S_{\!Q_+}(\mm,r+\nobreak 2\e\sqrt n)
\setminus S_{\!Q_+}(\mm,r-\nobreak 2\e\sqrt n) \,.
\end{align*}
For the latter inclusion we use that if $\nu$ is on the boundary of
$S_{\!Q_+}(\mm,r)$, then $\sum_{j\in \Q_+} \left( |\nu_j|
+2\e\right)^2 \leq r^2 +4\e \sum_{j\in Q_+} |\nu_j| + 4\e^2n \leq
r^2+4\e r\sqrt n + 4\e^2 n$, and similarly $\sum_{j\in \Q_+} \left(
|\nu_j| -2\e\right)^2\geq \left( r-2\e\sqrt n\right)^2$.

These estimates show that the asymptotic formula holds for families
$\mm \mapsto B_E \times S_{\!Q_+}(\mm,r)$, where $B_E$ is a box as
earlier, and $Q_-=\emptyset$. It works for constant radius~$r$, or
even for $r$ going down as a multiple of $\left(
\log\prod_{j\in Q_+}|\mm_j|\right)^{-\al}$ with $\al<\frac12$. As a
special case we obtain Proposition~\ref{prop-spheres}. There we have
required that all $\mm_j$ go to infinity, in order to have
$\npl\left(S_{\! Q_+}(\mm,r)\right) \sim 2^d \nv_1\left(S_{\!
Q_+}(\mm,r)\right) $.

\subsection{Slanted strips}Finally we consider, in the case $d=2$, a
strip of the form
\begin{equation} \tilde \Om_t = \left\{i(x,y)\;:\; t \leq x \leq 2t\,,
ax+b \leq y \leq ax+c\right\}\,,
\end{equation}
where $a,b,c\in \RR$, $a>0$, $c>b$, fixed. The parameter $t$ tends to
infinity. We take it such that $\tilde \Om_t$ is contained in
$\left( i[1,\infty) \right)^2$. Computations similar to those carried
out before give
\begin{align*}
\nv_1(\tilde \Om_t) &\sim \frac 73 a(c-b)t^3\,,\\
\nv_\r(\tilde \Om_t)&\ll \int_t^{2t} x^\r (c-b) (a x+c)^\r\, dx \ll
a^\r(c-b)t^{2\r+1}\,,
\\
\nv_1(\tilde \Om_t[2\e])&\ll \e t(2a t+b+c)(b-c+4\e) + \int_t^{2t} x
\e (a x+c)\, dx \\
&\ll _{a,b,c} \e t^2+ \e t^3 \ll \e t^3\,.
\end{align*}
We conclude that the asymptotic formula holds for~$\tilde \Om_t$, and
thus obtain Proposition~\ref{prop-slant}.


\newcommand\bibit[4]{
\bibitem {#1}#2: {\em #3;\/ } #4}

\end{document}